\documentclass[11pt, reqno]{article}
\usepackage{pstricks, pst-node}
\usepackage{amssymb}
\usepackage{latexsym}
\usepackage{amsmath}
\usepackage{amsthm}
\usepackage{amsfonts}
\usepackage{mathrsfs}
\usepackage{enumerate}
\usepackage{verbatim}
\usepackage[all,cmtip]{xy}
\usepackage[colorlinks=true,linkcolor=blue,citecolor=blue]{hyperref}

\newtheorem{theorem}{Theorem}[section]
\newtheorem*{theorem*}{Theorem}

\newtheorem*{thma}{Theorem A}
\newtheorem*{thmb}{Theorem B}

\newtheorem{lemma}[theorem]{Lemma}
\newtheorem*{lemma*}{Lemma}

\newtheorem{corollary}[theorem]{Corollary}
\newtheorem*{corollary*}{Corollary}

\newtheorem{proposition}[theorem]{Proposition}
\newtheorem*{proposition*}{Proposition}

\newtheorem{claim}[theorem]{Claim}
\newtheorem*{claim*}{Claim}
\newtheorem*{question}{Open Question}

\theoremstyle{definition}
\newtheorem{definition}[theorem]{Definition}

\newtheorem{remark}[theorem]{Remark}

\newcommand{\vl}{\mathcal L}

\newcommand{\vh}{\mathcal H}

\newcommand{\ind}{\textup{Ind}}
\newcommand{\spec}{\textup{spec}}
\newcommand{\diffeo}{\varPsi}
\newcommand{\supp}{\textup{supp}}

\newcommand{\cl}{\textup{C}\ell}
\newcommand{\bd}{A}
\newcommand{\lcn}{\langle\langle}
\newcommand{\rcn}{\rangle\rangle}

\begin{document}

\title{ A relative higher index theorem, diffeomorphisms and positive scalar curvature}
\author{Zhizhang Xie \and Guoliang Yu\footnote{Guoliang Yu was partially supported by the US National Science Foundation.} }
\date{}

\maketitle
\begin{abstract}
We prove a general relative higher index theorem for complete manifolds with positive scalar curvature towards infinity. We apply this theorem to study Riemannian metrics of positive scalar curvature on manifolds. For every two metrics of positive scalar curvature on a closed manifold and a Galois cover of the manifold, we define a secondary higher index class. Non-vanishing of this higher index class is an obstruction for the two metrics to be in the same connected component of the space of metrics of positive scalar curvature. In the special case where one metric is induced from the other by a diffeomorphism of the manifold, we obtain a formula for computing this higher index class. In particular, it follows that the higher index class lies in the image of the Baum-Connes assembly map.

\end{abstract}

\section{Introduction}
In this paper, we use methods from noncommutative geometry to study problems of positive scalar curvature on manifolds. From the work of Fomenko and Mischenko \cite{MF79}, Kasparov \cite{GK88}, and Connes and Moscovici \cite{CM90}, methods from noncommutative geometry have found many impressive applications towards geometry and topology, in particular, to those related to the Novikov conjecture and the positive scalar curvature problem. The fact that the positive scalar curvature problem is closely related to  the Novikov conjecture (or the Baum-Connes conjecture) was already made apparent by Rosenberg in \cite{JR83}. Block and Weinberger \cite{JBSW99}, and the second author \cite{GY98}\cite{GY00} successfully applied noncommutative geometric methods 
to determine the existence (nonexistence) of positive scalar curvature on certain classes of manifolds. By applying the work of Lott on higher eta invariants (which is noncommutative geometric) \cite{JLott92}, Leichtnam and Piazza studied the connectedness of the space of all Riemannian metric of positive scalar curvature on certain classes of manifolds \cite{ELPP01}. 

One of main tools used in all the studies mentioned above is index theory in the context of noncommutative geometry, often referred as \textit{higher index theory}. The method of applying (classical) index theory to study the positive scalar curvature problem on manifolds goes back to Lichnerowicz. By applying the Atiyah-Singer index theorem \cite{MAIS63},  he showed that a compact spin manifold does not support positive scalar curvature if its $\hat A$-genus is nonzero \cite{AL63}. With a refined version of the Atiyah-Singer index theorem \cite{MAIS71b}, Hitchin showed that half of the exotic spheres in dimension $1$ and $2$ ($\bmod 8$) cannot carry metrics of positive scalar curvature \cite{NH74}. This line of development was pursued further by Gromov and Lawson. In \cite{MGBL83}, they developed a relative index theorem and obtained nonexistence of positive scalar curvature for a large class of (not necessarily compact) manifolds. In \cite{UB95}, Bunke proved a relative higher index theorem and applied it to study problems of positive scalar curvature on manifolds.    

In this paper, we prove a general relative higher index theorem (for both real and complex cases). We apply this theorem to study Riemannian metrics of positive scalar curvature on manifolds. For every two metrics of positive scalar curvature on a closed manifold and a Galois cover of the manifold, there is a naturally defined secondary higher index class. Non-vanishing of this higher index class is an obstruction for the two metrics to be in the same connected component of the space of metrics of positive scalar curvature. In the special case where one metric is induced from the other by a diffeomorphism of the manifold, we obtain a formula for computing this higher index class. In particular, it follows that the higher index class lies in the image of the Baum-Connes assembly map.

 It is essential to allow real $C^\ast$-algebras and their (real) $K$-theory groups when studying problems of positive scalar curvature on manifolds, cf.\cite{MGBL83}\cite{JR83}. In fact, the (real) $K$-theory groups of real $C^\ast$-algebras provide more refined invariants for obstructions of existence of positive scalar curvature. We point out that the proofs in our paper are written in such a way that they apply to both the real and the complex cases. In order to keep the notation simple, we shall only prove the results for the complex case and indicate how to modify the arguments, if needed, for the real case. From now on, unless otherwise specified, all bundles and algebras are defined over $\mathbb C$.

Here is a synopsis of the main results of the paper. Let $X_0$ and $X_1$ be two even dimensional\footnote{In the real case, we assume $\dim X_0 = \dim X_1 \equiv 0 \pmod 8$.} spin manifolds with complete Riemannian metrics of positive scalar curvature (uniformly bounded below) away from compact sets. Assume that we have compact subspaces $K_i\subset X_i$ such that there is an (orientation preserving) isometry $\diffeo: \Omega_0 \to \Omega_1$, where  $\Omega_i \subset X_i - K_i$ is a union of (not necessarily all) connected components of $X_i - K_i$  (see Figure $\ref{fig:ends}$ in Section $\ref{sec:rel}$). We emphasize that the Riemannian metric on $\Omega_i$ may have \textit{nonpositive} scalar curvature on some compact subset.  Let $\mathcal S_i$ be the corresponding spinor bundle over $X_i$. We assume that  $\diffeo$ lifts to a bundle isometry $\widetilde \diffeo: \mathcal S_0|_{\Omega_0} \to \mathcal S_1|_{\Omega_1}$.  Let $(X_i)_\Gamma$ be a $\Gamma$-cover\footnote{All covering spaces considered in this paper are Galois covering spaces, i.e. regular covering spaces.} of $X_i$, where $\Gamma$ is a discrete group. We assume that $\diffeo$ lifts to an isometry on the covers. Let $D_i$ be the associated Dirac operator on $(X_i)_\Gamma$. Then we have $ D_1 = \widetilde \diffeo \circ D_0 \circ \widetilde \diffeo^{-1} $
on $(\Omega_1)_\Gamma$.

Let $N$ be a compact hypersurface in $\Omega\cong \Omega_i$ such that $N$ cuts $X_i$ into two components. We separate off the component that is inside $\Omega_i$ and denote the remaining part of  $X_i$ by $Y_i$ (see Figure $\ref{fig:ends}$ in Section $\ref{sec:rel}$). We obtain $X_2$ by gluing $Y_0$ and $Y_1$ along $N$. Moreover, we glue the spinor bundles over $Y_0$ and  $Y_1$ to get a spinor bundle over $X_2$. All these cutting-pastings lift to the covers, and produce a $\Gamma$-cover $(X_2)_\Gamma$ of $X_2$. Let $D_2$ be the associated Dirac operator on $(X_2)_\Gamma$.  For each $D_i$, we have its higher index class $\ind(D_i) \in K_0(C_r^\ast(\Gamma))$ (resp. $K_0(C_r^\ast(\Gamma; \mathbb R))$ in the real case). We have the following relative higher index theorem (Theorem \ref{thm:rit}).

\begin{thma}
\[ \ind(D_2)  = \ind(D_0) - \ind(D_1). \]

\end{thma}
  
The usefulness of the above theorem lies in the fact that the index class on the left hand side is computable in many cases (for example, when $X_2$ is compact), while the index classes on the right hand side are difficult to compute. In the proof of this theorem, we carry out a construction of invertible doubles (Theorem $\ref{thm:double}$). Our construction takes place on manifolds with $C^\ast_r(\Gamma)$-bundles\footnote{For a real or complex $C^\ast$-algebra $\mathcal A$, by a $\mathcal A$-bundle over a manifold $M$, we mean a locally trivial Banach vector bundle over $M$ whose fibers have the structure of finitely generated projective $\mathcal A$-modules.} (resp. $C^\ast_r(\Gamma; \mathbb R)$-bundles ) and  generalizes the invertible double construction for manifolds with classical vector bundles (i.e. $\mathbb C$-vector bundles or $\mathbb R$-vector bundles), cf. \cite[Chapter 9]{BBKW93}. We point out that if the scalar curvature on $\Omega$ is positive everywhere, then our theorem above follows from Bunke's relative higher index theorem \cite[Theorem 1.2]{UB95}. In our theorem, we allow the scalar curvature on $\Omega$ to be nonpositive on a compact subset. In particular, in the case when $X_0$ and $X_1$ are both compact, our theorem applies to any Riemmannian metrics  (possibly with scalar curvature nowhere positive) on $X_0$ and $X_1$. 

 As an application of our relative higher index theorem, we consider a compact odd dimensional\footnote{In the real case, we assume $\dim M \equiv -1 \pmod 8$.} spin manifold $M$ (with a fixed spin structure) which supports positive scalar curvature. Let $M_\Gamma$ an $\Gamma$-cover of $M$, where $\Gamma$ is a discrete group. Let $\diffeo$ be an orientation preserving diffeomorphism $\diffeo: M \to M$ which in addition preserves the spin structure of $M$. Choose $g_0\in \mathcal R^+(M)$. Let $g_1 = (\diffeo^{-1})^\ast g_0$ and $g_t$  a smooth path of Riemannian metrics on $M$ with
\[ 
g_t = \begin{cases}
g_0 \quad \textup{for $t\leq 0$, }\\
g_1 \quad \textup{for $t\geq 1$,} \\
\textup{any smooth homotopy from $g_0$ to $g_1$ for $0\leq t \leq 1$.}
\end{cases}\] 
Then $X = M\times \mathbb R$ endowed with the metric $ h = g_t + (dt)^2$ becomes a complete Riemannian manifold with positive scalar curvature away from a compact set. Let $X_\Gamma = M_\Gamma \times \mathbb R$ and  $D$ the corresponding Dirac operator on $X_\Gamma$. Then we have the higher index class $\ind(D)\in K_0(C_r^\ast(\Gamma))$ (resp.  $\ind(D)\in K_0(C_r^\ast(\Gamma; \mathbb R))$ in the real case).

Assume that $\diffeo$ lifts to a diffeomorphism $\widetilde \diffeo: M_{\Gamma} \to M_\Gamma$. This is always the case when $\Gamma = \pi_1(M)$ the fundamental group of $M$. Notice that $\diffeo$ induces an automorphism $\Gamma \to \Gamma$. Let $\Gamma\rtimes\mathbb Z$ be the semi-direct product with the action of $\mathbb Z$ on $\Gamma$ induced by $\diffeo$. Then $X_\Gamma = M_\Gamma \times \mathbb R$ becomes a  $(\Gamma\rtimes \mathbb Z)$-cover of $M_\diffeo$. Here $M_\diffeo = (M\times [0, 1])/\sim,$
 where $\sim$ is the equivalence relation $(x, 0)\sim (\diffeo(x), 1)$ for $x\in M$. We denote by $D_{\Gamma\rtimes \mathbb Z}$ the Dirac operator on $X_\Gamma$, which defines a higher index class $\ind(D_{\Gamma\rtimes \mathbb Z}) \in K_0(C_r^\ast(\Gamma\rtimes \mathbb Z))$ (resp.  $\ind(D_{\Gamma\rtimes \mathbb Z}) \in K_0(C_r^\ast(\Gamma\rtimes \mathbb Z; \mathbb R))$ in the real case ), cf. \cite[Section 5]{CM90}.  Now let $\iota:  \Gamma \hookrightarrow \Gamma\rtimes \mathbb Z$ 
 be the natural inclusion map, which induces a homomorphism $\iota_\ast : K_0(C_r^\ast(\Gamma)) \to K_0(C_r^\ast(\Gamma\rtimes \mathbb Z)).$
Then we have the following theorem. 	
\begin{thmb} 
 \[  \iota_\ast( \ind(D) ) = \ind(D_{\Gamma\rtimes \mathbb Z})  \]
in $ K_0(C_r^\ast(\Gamma\rtimes \mathbb Z))$ (resp.  $K_0(C_r^\ast(\Gamma\rtimes \mathbb Z; \mathbb R))$ for the real case).
\end{thmb} 
 If we assume that the strong Novikov conjecture holds for $\Gamma\rtimes \mathbb Z$, then the above theorem provides a formula to determine when $\iota_\ast( \ind(D) )$ is nonvanishing. Note that the above theorem implies that  $\iota_\ast( \ind(D) )$ lies in the image of the Baum-Connes assembly map $\mu:  K_0^{\Gamma\rtimes \mathbb Z}(\underline{E}(\Gamma\rtimes\mathbb Z))\to K_0(C_r^\ast(\Gamma\rtimes \mathbb Z))$ (or its real analogue \cite{PBMK04}). It remains an open question whether $\ind(D) \in K_0(C_r^\ast(\Gamma)) $  lies in the image of the Baum-Connes assembly map $\mu:  K_0^{\Gamma}(\underline{E}\Gamma)\to K_0(C_r^\ast(\Gamma))$. 

An outline of the paper is as follows. In Section $\ref{sec:pre}$, we review some basic facts in $K$-theory and index theory. In Section $\ref{sec:dirac}$, we discuss some basic properties of Dirac operators (in Hilbert modules over a $C^\ast$-algebra) and construct their higher index classes (with finite propagation property). In Section $\ref{sec:rel}$, we prove a general relative higher index theorem. In Section $\ref{sec:id}$, we carry out an invertible double construction. In Section $\ref{sec:diffeo}$, we apply our relative higher index theorem to study positive scalar curvature problem on manifolds under diffeomorphisms.

\noindent\textbf{Acknowledgements} We want to thank Matthias Lesch and Rufus Willett for many stimulating discussions and helpful comments. We also thank Etienne Blanchard and Ulrich Bunke for helpful comments. 

\section{Preliminaries}\label{sec:pre}
In this section, we review some preliminary facts in $K$-theory and index theory. 
\subsection{Abstract index theory}

Let $\mathcal B$ be a unital $C^\ast$-algebra and  
\[  0\to \mathcal I \to \mathcal B\to \mathcal B/\mathcal I  \to 0 \] a short exact sequence of $C^\ast$-algebras. Then we have the following six-term\footnote{In the real case, the long exact sequence has $24$ terms, due to the fact real $K$-theory has periodicity $8$.} long exact sequence in $K$-theory.

\[ \xymatrix { K_0(\mathcal I) \ar[r] & K_0(\mathcal B) \ar[r] & K_0(\mathcal B/\mathcal I )  \ar[d]^{\ind}\\
K_1(\mathcal B/\mathcal I  ) \ar[u]^{\ind}
  & K_1(\mathcal B) \ar[l] & K_1(\mathcal I) \ar[l]
}
\]

In particular, each invertible element in $u\in \mathcal B/\mathcal I$ defines an element $\ind(u)\in K_0(\mathcal I)$, called the index of $u$.

We recall an explicit construction of the index map (in the even case). Let $u\in \mathcal B$ such that $u$ becomes invertible in $\mathcal B/\mathcal I$. Then there exists $v\in \mathcal B$ such that both $uv -1$ and $vu-1$ are in $\mathcal I$. We define 
\[ w = \begin{pmatrix} 1 & u \\ 0 & 1\end{pmatrix} \begin{pmatrix} 1 & 0 \\ -v & 1\end{pmatrix} \begin{pmatrix} 1 & u  \\ 0 & 1 \end{pmatrix}\begin{pmatrix} 0 & -1\\ 1 & 0 \end{pmatrix}. \]
Notice that $w$ is invertible and a direct computation shows that 
\[  w - \begin{pmatrix} u & 0 \\ 0 & v \end{pmatrix} \in M_2(\mathcal I) = M_2(\mathbb C)\otimes \mathcal I.\]
Consider the idempotent
\[ p = w \begin{pmatrix} 1 & 0 \\ 0 & 0\end{pmatrix} w^{-1} = \begin{pmatrix} uv + uv(1-uv) & (2 + uv)(1-uv) u \\ v(1-uv) & (1-uv)^2\end{pmatrix}. \]
We have  
\[ p - \begin{pmatrix} 1 & 0 \\0 & 0\end{pmatrix} \in M_2(\mathcal I) \]
and 
\[  \ind(u) = [p] - \left[\begin{pmatrix} 1 & 0 \\0 & 0\end{pmatrix} \right] \in K_0(\mathcal I).\]

Now suppose $F = \begin{pmatrix} 0 & u^\ast \\ u & 0\end{pmatrix} \in M_2(\mathcal B)$ is a self-adjoint element such that $F^2  - 1 \in M_2(\mathcal I)$.
In this case, the index of $F$ is defined to be \[ \ind(F) := \ind(u).\]

\subsection{K-theory and almost idempotents}\label{subsec:k}

There are several equivalent ways to define the $K$-theory groups of a $C^\ast$-algebra. In this subsection, we review a slightly non-standard definition, which will be used in the later sections.

In the following discussion, we fix a (sufficiently small) constant $\tau > 0$  once and for all.  
 
\begin{definition}\label{def:close}

Let $\mathcal B$ be a $C^\ast$-algebra and $\mathcal I$ a closed ideal of $\mathcal B$. 
\begin{enumerate}[(a)]

\item  We call an element $ x\in \mathcal B$ is $\tau$-close to $\mathcal I$ if there exists an element $y\in \mathcal I^+$ such that 
\[   \| x- y\| < \min\{ \tau, \|x\|^{-1}\tau\}. \]

\item An element $z \in \mathcal B$ is called a $\tau$-almost idempotent if 
\[  \| z^2 - z\|< \tau.   \]
\end{enumerate}
\end{definition}

Notice that if $z$ is a $\tau$-almost idempotent, then (as long as $\tau$ is sufficiently small) we can choose disjoint open sets $U$ and $V$ such that $\spec(z)\subset U\cup V$ with $0\in U$ and $1\in V$. Set $h = 0 $ on $U$ and $h= 1$ on $V$ and define
\[  p = \frac{1}{2\pi i}\int_{C} \frac{h(\zeta)}{\zeta - z} d\zeta \]
where $C$ is a contour surrounding $\spec(z)$ in $U\cup V$. Then  $p$ is an idempotent and therefore defines a $K$-theory class. 

Now let $p$ be an idempotent in $M_\infty(\mathcal B)$. In general, such an idempotent does not define a $K$-theory class in $K_0(\mathcal I)$. However,  if $p$ is $\tau$-close to $\mathcal I$ (with $\tau$ sufficiently small), then $p$ does uniquely define an element in $K_0(\mathcal I)$. Indeed, choose $q\in M_\infty (\mathcal I^+)$ so that $\| p -q \| < \min\{\tau, \|p\|^{-1}\tau\}$. Since $p^2 - p = 0$, we have 
\[  \| q^2 - q\|  \leq  \| (p-q)p\| + \|q(p-q)\| + \|p-q\| < 4\tau,    \]    
i.e. $q\in M_\infty( \mathcal I^+)$ is a $(4\tau)$-almost idempotent. By the above discussion,  $q$ defines a $K$-theory class in $K_0(\mathcal I^+)$. Let $\pi$ be the  quotient map 
\[  \pi : \mathcal I^+ \to \mathcal I^+/\mathcal I = \mathbb C. \]
and $\pi_\ast: K_0(\mathcal I^+)\to K_0(\mathbb C) = \mathbb Z$ the induced map on $K_0$-groups. If $\pi_\ast([q]) = [n]$, then we have $[q] - [n] \in K_0(\mathcal I)$. We shall still denote this class by $[p] -[n]$ if no confusion arises. 


\begin{remark}
Notice that the class $[p] - [n]\in K_0(\mathcal I)$ does not depend on the choice of $q$. Indeed, if we choose another $q'\in M_\infty (\mathcal I^+)$ such that $\| p -q' \| < \min\{\tau, \|p\|^{-1}\tau\}$, then we have $\| q- q'\|< 2 \min\{\tau, \|p\|^{-1}\tau\}$. It is easy to verify that $[q] = [q']\in K_0(\mathcal I^+)$.    
\end{remark}

\subsection{A difference construction}\label{sec:diff}
In this subsection, we review the difference construction in $K$-theory from \cite[Section 6]{GKGY06}. Let $\mathcal B$ be a $C^\ast$-algebra and $\mathcal I$ be a two-sided closed ideal in $\mathcal B$. For each pair of idempotents $p, q\in \mathcal B$ with $p-q\in \mathcal I$, we shall define a difference element $E(p, q) \in K_0(\mathcal I)$.  

First consider the invertible element
\[ Z(q) = \begin{pmatrix} q & 0 & 1-q & 0 \\ 1-q & 0 & 0 & q \\ 0 & 0 & q  & 1-q \\ 0 & 1 & 0 & 0 \end{pmatrix} \]
whose inverse is 
\[ Z(q)^{-1} = \begin{pmatrix} q & 1-q & 0 & 0 \\ 0 & 0 & 0 & 1 \\ 1-q & 0 & q  & 0 \\ 0 & q & 1-q & 0 \end{pmatrix}. \]
Then we define 
\[ E_0(p, q) = Z(q)^{-1} \begin{pmatrix} p & 0 & 0 & 0 \\ 0 & 1-q & 0 & 0 \\ 0 & 0 & 0  & 0 \\ 0 & 0 & 0 & 0 \end{pmatrix} Z(q). \]
A direct computation shows that 
\begin{align}
E_0(p, q)   & = \begin{pmatrix} 1 + q(p-q) q & 0 & qp(p-q) & 0 \\ 0 & 0 & 0 & 0 \\ (p-q)pq & 0 & (1-q)(p-q)(1-q)  & 0 \\ 0 & 0 & 0 & 0 \end{pmatrix}. \label{eq:diff}
 \end{align}

It follows immediately that $E_0(p, q) \in M_4(\mathcal I^+)$ and  $E_0(p, q) = e$ modulo $M_4(\mathcal I)$, where 
\[  e = \begin{pmatrix} 1 & 0 & 0 & 0 \\ 0 & 0 & 0 & 0 \\ 0 & 0 & 0  & 0 \\ 0 & 0 & 0 & 0 \end{pmatrix}.   \]
\begin{definition} Define
\[ E(p, q) = [E_0(p, q) ] - [e] \in K_0(\mathcal I). \]
\end{definition}

\begin{remark}
In fact, the same construction works when $(p-q)$ is $\tau$-close to $\mathcal I$. In this case,  although $E_0(p, q) -e \notin M_4(\mathcal I) $, the explicit formula $\eqref{eq:diff}$ shows that $E_0(p, q)$ is $\tau$-close to $\mathcal I$ (with a slight modification of the definition of $\tau$-closeness). According to the discussion in the previous subsection,  $E_0(p, q)$ defines a $K$-theory class in $K_0(\mathcal I)$, which we shall still denote by $E(p, q)$.  
\end{remark}

\section{Dirac operators and their higher index classes}\label{sec:dirac}
In this section, we review some basic properties of Dirac operators over Galois covers of complete manifolds and their higher index classes. 

Let $X$ be a complete even dimensional spin manifold. Let $X_\Gamma$ be a $\Gamma$-cover of $X$. We define a flat $C_r^\ast(\Gamma)$-bundle\footnote{In the real case, we consider the bundle $ \mathcal V = X_\Gamma \times_\Gamma C_r^\ast(\Gamma; \mathbb R)$.} $\mathcal V$ on $X$ by
\[ \mathcal V = X_\Gamma \times_\Gamma C_r^\ast(\Gamma),  \]
where $\Gamma$ acts on $X_{\Gamma}$  and $C^\ast_r(\Gamma)$ diagonally. Denote by  $\mathcal S = \mathcal S^+ \oplus \mathcal S^-$ the spinor bundle over $X$ and 
\[  \varepsilon = \begin{pmatrix} 1 & 0 \\ 0 & -1\end{pmatrix} \] 
the grading operator on $\mathcal S$. Then with the flat connection of $\mathcal V$, we can define the Dirac operator 
\[  D_{\mathcal V} : \Gamma^\infty (X, \mathcal S\otimes \mathcal V) \to \Gamma^\infty (X, \mathcal S\otimes \mathcal V) \]
where $\Gamma^\infty(X, \mathcal S\otimes \mathcal V)$ is the space of smooth sections of $\mathcal S\otimes \mathcal V$ over $X$. We will simply write $D$ for $D_{\mathcal V}$ if no ambiguity arises.   With the $\mathbb Z_2$-grading on $\mathcal S$, we have  
\[  D = \begin{pmatrix} 0 & D^-  \\ D^+ & 0 \end{pmatrix}. \] 
Here are some standard properties of this Dirac operator $D$ as a unbounded operator on the $C^\ast_r(\Gamma)$-Hilbert module $\vl^2(X, \mathcal S\otimes \mathcal V)$ the space of $L^2$ sections of $\mathcal S\otimes \mathcal V$ over $X$:
\begin{enumerate}[(a)]
\item $D$ is an essentially self-adjoint operator; 
\item $D^2 \sigma  = 0  \Leftrightarrow D \sigma = 0$, for $\sigma \in \ \vl^2(X, \mathcal S\otimes \mathcal V)$; 
\item if the scalar curvature $\kappa$ of the manifold $X$ is uniformly bounded, then the maximal domain of $D$ on $\vl^2(X, \mathcal S \otimes \mathcal V)$ is exactly the Soblev space $\vh^1(X, \mathcal S\otimes \mathcal V)$; 
\item  $  \displaystyle D^2 = \nabla^\ast\nabla + \frac{\kappa}{4}$,  where $\nabla: \Gamma^\infty (X, \mathcal S\otimes \mathcal V) \to \Gamma^\infty (X, T^\ast X\otimes \mathcal S\otimes \mathcal V)$ is the connection on the bundle $\mathcal S\otimes \mathcal V$ and $\nabla^\ast$ is the adjoint of $\nabla$. 

\end{enumerate}


From now on, let us assume that $X$ has \textit{(strictly) positive scalar curvature towards infinity}. More precisely, there exist a subset $\Omega\subset X$ such that $X - \Omega$ is compact and the Riemannian metric has positive scalar curvature $>k_0$ on $\Omega$, for some positive constant $k_0$. 

In this case, there exists a compactly supported function $\rho \geq 0$ on $X$ such that 
\begin{enumerate}[(i)]
\item $\displaystyle \frac{\kappa}{4} + \rho^2 \geq c_0 > 0$ for some fixed constant $c_0$,
\item $\| [D, \rho] \|$ as small as we want, in particular $ \displaystyle  < \frac{c_0}{2}$.
\end{enumerate}

\begin{lemma}\label{lemma:lb}
 $D_{\mathcal V} + \varepsilon\rho : \vh^1(X, \mathcal S\otimes \mathcal V) \to \vl^2(X, \mathcal S\otimes \mathcal V) $ is bounded below. 
\end{lemma}
\begin{proof}
Indeed, for each $\sigma \in \vh^1(X, \mathcal S\otimes \mathcal V)$, we have
\begin{align*}
 \| (D_{\mathcal V} + \varepsilon \rho)\sigma \|^2_0  & =\| \langle (D +\varepsilon \rho)\sigma, (D + \varepsilon \rho) \sigma\rangle \| \\
 & =  \| \langle (D^2 + [D, \rho] \varepsilon +  \rho^2)\sigma, \sigma\rangle \|\\
 & \geq \| \langle (\nabla^\ast\nabla +\frac{\kappa}{4} + \rho^2)\sigma, \sigma\rangle\| - \frac{c_0}{2}\|\sigma\|_0^2 \\
 & =  \| \langle \nabla\sigma, \nabla\sigma\rangle + \langle (\frac{\kappa}{4} + \rho^2)\sigma, \sigma\rangle  \|  - \frac{c_0}{2}\|\sigma\|_0^2 \\
 & \geq c \| \langle \nabla\sigma, \nabla\sigma\rangle + \langle \sigma, \sigma \rangle \| =  c \|\sigma\|^2_{\vh^1}
\end{align*}
where $\|\cdot \|_0$ denotes the $L^2$-norm on $\vl^2(X, \mathcal S\otimes \mathcal V) $  and $   c =  \min\{1/2 , c_0/2 \} > 0.$ We have used the fact that $ a \geq b \geq 0\Rightarrow \|a\|\geq \|b\| $ in a $C^\ast$-algebra.

\end{proof} 
It follows that $D_{\mathcal V}+ \varepsilon \rho$ has a bounded inverse 
\[ (D_{\mathcal V} +\varepsilon \rho)^{-1} : \vl^2(X, \mathcal S\otimes \mathcal V) \to \vh^1(X, \mathcal S\otimes \mathcal V). \]

\subsection{Generalized Fredholm operators in Hilbert modules}

Let $\mathcal A$ be a $C^\ast$-algebra and $\vh_{\mathcal A}$ a Hilbert module over $\mathcal A$. We denote the space of all adjointable operators in $\vh_{\mathcal A}$ by $\mathcal B (\vh_{\mathcal A})$. Let $\mathcal K(\vh_{\mathcal A})$ be the space of all compact adjointable operators in $\vh_{\mathcal A}$. Note that $\mathcal K(\vh_{\mathcal A})$ is a two-sided ideal in $\mathcal B(\vh_{\mathcal A})$. 

\begin{definition}(cf. \cite[Chapter~17]{WO93}) 
An operator $F\in \mathcal B(\vh_{\mathcal A})$ is called a generalized Fredholm operator if $\pi(F)\in \mathcal B(\vh_{\mathcal A})/\mathcal K(\vh_{\mathcal A})$ is invertible.
\end{definition}

\begin{proposition}\label{prop:fred}
Let 
\[ F = D_{\mathcal V}(D_{\mathcal V}^2 + \rho^2 + [D_{\mathcal V}, \rho]\varepsilon)^{-1/2} :  \vl^2(X, \mathcal S\otimes \mathcal V) \to  \vl^2(X, \mathcal S\otimes \mathcal V). \] Then $F$ is a generalized Fredholm operator.  
\end{proposition}
\begin{proof}
Let  $D = D_{\mathcal V}$ and 
\[ T = ( D_{\mathcal V}^2 + \rho^2 + [D_{\mathcal V}, \rho]\varepsilon)^{-1/2} = \sqrt {\frac{1}{(D_{\mathcal V} + \varepsilon \rho)^2}}. \]
First, we show that $F : \vl^2(X, \mathcal S\otimes \mathcal V)\to \vl^2(X, \mathcal S\otimes \mathcal V)$  is bounded. Indeed, 
\begin{align*}
\|F \sigma\|_0^2 &= \|\langle DT\sigma,  DT \sigma \rangle \| \\
& = \|\langle (TD^2T \sigma,   \sigma \rangle \|\\
& = \|\left\langle \sigma - T(\rho^2 + [D, \rho]\varepsilon) T \sigma, \ \sigma \right\rangle \|\\
& \leq \| 1 - R\| \|\sigma\|_0^2
\end{align*}
where $ R= T(\rho^2 + [D, \rho]\varepsilon) T:  \vl^2(X, \mathcal S\otimes \mathcal V) \to  \vl^2(X, \mathcal S\otimes \mathcal V) $ is compact, in particular bounded. 

To show that $F$ is a generalized Fredholm operator, it suffices to show that $F^2 - 1$ is compact, i.e. $F^2 - 1 \in \mathcal K(\mathcal H_{C_r^\ast(\Gamma)})$. Since 
\[ \frac{1}{\sqrt x} = \int_0^\infty \frac{1}{x^2 + \lambda^2}  d\lambda, \]
we have 
\[ T =  \sqrt {\frac{1}{(D + \varepsilon \rho)^2}} = \int_0^\infty ((D + \varepsilon\rho)^2 + \lambda^2)^{-1} d\lambda. \]
Now notice that 
\begin{align*}
 &((D + \varepsilon\rho)^2+ \lambda^2)^{-1} D - D ((D + \varepsilon\rho)^2 + \lambda^2)^{-1} \\
  & = ((D + \varepsilon\rho)^2 + \lambda^2)^{-1}  [D, \rho^2 + [D, \rho]\varepsilon ] ((D + \varepsilon\rho)^2 + \lambda^2)^{-1}.
\end{align*}
It follows that 
\begin{align*}
F^2 & = DTDT \\
& = D \left(\int_0^\infty ((D + \varepsilon\rho)^2 + \lambda^2)^{-1} d\lambda \right) D T\\
& = \left( D^2 \int_0^\infty ((D + \varepsilon\rho)^2 + \lambda^2)^{-1} d\lambda +  D \int_0^\infty K(\lambda) d\lambda \right) T \\
& = D^2 T^2 + D K T	\\
& = 1 - ( \rho^2 + [D, \rho]\varepsilon) T^2 + D K T^2
\end{align*} 
where 
\[ K(\lambda) =((D + \varepsilon\rho)^2 + \lambda^2)^{-1}  [D, \rho^2 + [D, \rho]\varepsilon ] ((D + \varepsilon\rho)^2 + \lambda^2)^{-1}\]
 and  $K = \int_0^\infty K(\lambda) d\lambda$. Since $\rho$ and $[D, \rho]$ have compact supports, it follows that   $( \rho^2 + [D, \rho]\varepsilon) T^2 $ and $D K T^2$ are both compact. This finishes the proof.

\end{proof}

\subsection{Finite propagation speed}\label{subsec:fps}

In this subsection, we shall show that for any first order essentially selfadjoint differential operator $D : \Gamma^\infty (X, \mathcal S\otimes \mathcal V) \to \Gamma^\infty (X, \mathcal S\otimes \mathcal V)$, if the propagation speed of $D$ is finite, that is, 
\[ c_D = \sup\{ \|\sigma_D(x, \xi)\|: x\in X, \xi\in T_x^\ast X, \|\xi\| = 1 \} < \infty, \]
then the unitary operators $e^{isD}$ satisfy the following finite propagation property. The results in this subsection are straightforward generalizations of their corresponding classical results. We refer the reader to \cite[Section~10.3]{NH-JR00} for detailed proofs.  

In the rest of this subsection, let us fix a closed  (not necessarily compact) subset $K\subset X$. We denote  
 \[   Z_\beta  = Z_\beta(K) = \{ x\in X\mid d(x, K) < 2\beta \} \]
 where $d(x, K)$ is the distance of $x$ from $K$ and $\beta > 0$ is some fixed constant.

\begin{proposition}\label{prop:fps}
Let 
\[ D :\Gamma^\infty (X, \mathcal S\otimes \mathcal V) \to \Gamma^\infty (X, \mathcal S\otimes \mathcal V)\]
 be a first order essentially self-adjoint differential operator on a complete Riemannian manifold $X$. Suppose $D$ has finite propagation speed $c_D < \infty$. Then for all $\sigma\in \Gamma^\infty (X, \mathcal S\otimes \mathcal V) $ supported within $K$,  we have    $e^{isD} \sigma$ supported in $Z_\beta(K)$, for all $s$ with $|s| < \beta c_D^{-1}$.
  
\end{proposition}

\begin{corollary}\label{cor:fp}
Let $\varphi$ be a bounded Borel function on $\mathbb R$ whose Fourier transform is supported in $(-\beta c_D^{-1}, \beta c_D^{-1})$. If $\sigma \in \Gamma^\infty (X, \mathcal S\otimes \mathcal V)$ is supported in $ K$, then $\varphi(D ) \sigma$ is supported in $Z_\beta(K)$.     
\end{corollary}

\begin{proof}
Since
\[ \langle \varphi(D)\sigma, 
\nu \rangle  = \frac{1}{2\pi} \int_{-\infty}^\infty \langle e^{i sD} \sigma, \nu\rangle \widehat \varphi(s) ds, \]
the corollary follows immediately from the proposition above.

\end{proof}
\begin{corollary}
Let $D_1$ and $D_2$ be essentially selfadjoint differential operators on $X$ which coincide on $Z_\beta(K)$. Suppose $\varphi$ is a bounded Borel function on $\mathbb R$ whose Fourier transform is supported in $(-\beta c_D^{-1}, \beta c_D^{-1})$. Then 
\[ \varphi(D_1) \sigma = \varphi(D_2)\sigma \]
for all $\sigma \in \Gamma^\infty (X, \mathcal S\otimes \mathcal V)$ supported in $K$.  
\end{corollary}

\subsection{Higher index classes}\label{sec:hic}
In this subsection, we construct the higher index class $\ind(D_{\mathcal V})$ (with finite propagation property) of the Dirac operator $D_{\mathcal V}$.

Let $D = D_{\mathcal V}$. A similar argument as that in Proposition $\ref{prop:fred}$ shows that  
\[ G = D(D^2 + \rho^2)^{-1/2}  \] 
is  a generalized Fredholm operator.  With respect to the $\mathbb Z_2$ grading, 
\[  G = \begin{pmatrix} 0 & U_G^\ast \\ U_G & 0 \end{pmatrix} = \begin{pmatrix} 0 & D^- (D^+D^- + \rho^2)^{-1/2} \\ D^+(D^-D^+ + \rho^2) ^{-1/2} & 0 \end{pmatrix}. \]
Then the index of $D$ is
\[ \ind(D) := \ind(U_G) \in  K_0(\mathcal K(\mathcal H_{C_r^\ast(\Gamma)}) ) \cong K_0(C^\ast_r(\Gamma)). \]

Such a representative of the index class does not have finite propagation property in general. In the following discussion, we shall carry out an explicit construction to remedy this.  

Before getting into the details, we would like to point out that one can also use the operator 
\[ F = D (D^2 + \rho^2 + [D, \rho]\varepsilon)^{-1/2}\]
to construct a representative of the index class $\ind(D)$. This requires some justification, since after all $F$ is not an odd operator with respect to the $\mathbb Z_2$-grading. In order to define an index class,  we need to take the odd part of $F$.  If we write 
\[ F  = \begin{pmatrix} A & U_F^\ast \\ U_F & C \end{pmatrix},  \]
then its odd part is 
\[ \begin{pmatrix} 0 & U_F^\ast \\ U_F & 0 \end{pmatrix}.  \]
One readily verifies that 
\[  G- F = \int_0^\infty D (D^2 + \rho^2 + [D, \rho]\varepsilon + \lambda^2)^{-1} ([D, \rho]\varepsilon) (D^2 + \rho^2 + \lambda^2)^{-1} d\lambda.\]
It follows that we can choose an appropriate $\rho$ so that $F$ is sufficiently close to $G$. In particular, $U_F$ is sufficiently  close to $U_G$. Since $U_G$ is generalized Fredholm, it follows that $U_F$ is also generalized Fredholm.  Moreover, 
\[  \ind(D) = \ind(U_G) = \ind(U_F). \]
In fact, the operator 
\[ F' := (D + \varepsilon\rho) (D^2 + \rho^2 + [D, \rho]\varepsilon)^{-1/2}\]
also produces the same index class for $D$, since $F'$ only differs from $F$ by a compact operator $\varepsilon\rho (D^2 + \rho^2 + [D, \rho]\varepsilon)^{-1/2}$. 

Notice that neither $F$ nor $F'$ produces an index class of finite propagation property yet. In the following, we shall  approximate $F'$ by an operator of finite propagation property and use the latter to construct a representative of the index class $\ind(D)$. 

\begin{definition}
A smooth function $\chi : \mathbb R \to [-1, 1]$ is a normalizing function if 
\begin{enumerate}[(1)]
\item $\chi (-\lambda) = -\chi(\lambda) $ for all $\lambda\in\mathbb R$,
\item $\chi(\lambda) > 0$ for $\lambda>0$,
\item $\chi(\lambda) \to \pm 1$ as $\lambda\to \pm \infty$. 
\end{enumerate}
\end{definition}


Since $D+ \varepsilon \rho$ is invertible, there exists  a normalizing function $\chi$ such that 
\[  \begin{cases}
 \chi(\lambda) = 1 & \textup{for $\lambda \geq a$} \\
\chi(\lambda) = -1 & \textup{for $\lambda \leq -a$ }
 \end{cases}\]
where  $a>0$ is a constant such that $\spec (D + \varepsilon \rho)  \cap (-a, a) = \emptyset$. In fact, 
\[  \chi(D+ \varepsilon\rho) = F' =  (D + \varepsilon\rho) (D^2 + \rho^2 + [D, \rho]\varepsilon)^{-1/2}.   \]

\begin{lemma}\label{lemma:normal}
For any $\delta >0$, there exists a normalizing function $\varphi$, for which its distributional Fourier transform $\widehat \varphi$ is compactly supported and for which $\xi\widehat \varphi(\xi)$ is a smooth function, such that
\[  \sup_{\lambda \in \mathbb R} |\varphi(\lambda) -\chi(\lambda)|< \delta.\]

\end{lemma}
\begin{proof}
By our explicit choice of $\chi$, we see that $\chi'$  has compact support. Therefore  $ \xi \widehat \chi(\xi)  = \widehat{\chi'} (\xi)$ is a smooth function.

Let $m$ be a smooth even function on $\mathbb R$ whose Fourier transform is a compactly supported smooth function. Moreover, we assume that 
\[ \int_{\mathbb R} m(\lambda) \, d\lambda  = 1.\]
Define $m_t(\lambda)  = t^{-1} m(t^{-1} \lambda).$  It is easy to verify that  $(m_t \ast \chi)$ is a normalizing function and its distributional Fourier transform
\[ \widehat{m_t\ast \chi} = \widehat m_t \cdot \widehat \chi \]
is compactly supported. Moreover,  $  m_t \ast \chi \to \chi $ uniformly as $t \to 0 $, since $\chi$ is uniformly continuous on $\mathbb R$. Now let $\varphi = m_t\ast \chi$ for some sufficiently small $t>0$. Notice that $\xi\widehat \varphi(\xi)$ is a smooth function, since both $\widehat m_t$ and $\xi\widehat \chi(\xi)$ are smooth functions. This finishes the proof. 

\end{proof}

\begin{definition} Define $F_\varphi = \varphi(D+ \varepsilon \rho)$.

\end{definition}

Since the distributional Fourier transform $\widehat \varphi$ of $\varphi$ has compact support, it follows immediately from Corollary $\ref{cor:fp}$ that $ F_\varphi $ has finite propagation property. More precisely, we have the following lemma.

\begin{lemma} 
Suppose 
\[  \supp (\widehat \varphi ) \subset (-b, b). \]
Let $\beta =  c_D \cdot b $, where $c_D$ is the propagation speed  of $(D + \varepsilon \rho)$.
Let $K$ be a closed subset of $X$ and 
 \[   Z_\beta  = \{ x\in X\mid d(x, K) < 2\beta \} \]
 where $d(x, K)$ is the distance of $x$ from $K$. If $\sigma \in \Gamma^\infty (X, \mathcal S\otimes \mathcal V)$ is supported within $K$, then $\varphi(D+ \varepsilon \rho) \sigma$ is supported in $Z_\beta$. 
\end{lemma}

\noindent We denote the odd-grading part of $F_\varphi$ by  $\begin{pmatrix} 0 & U^\ast \\ U & 0 \end{pmatrix}$
and denote the odd-grading part of $F'$ by $ \begin{pmatrix} 0 & U_{F'}^\ast \\ U_{F'} & 0 \end{pmatrix},$
where $ F' = \chi(D+ \varepsilon\rho) =  (D + \varepsilon\rho) (D^2 + \rho^2 + [D, \rho]\varepsilon)^{-1/2}.$
By choosing $\varphi$ sufficiently close to $\chi$, we can make   $\| F_\varphi - F'\| $
as small as we want. 
In particular, we can choose $\varphi$ such that  $U$ is sufficiently close to $U_{F'}$.

To summarize, we have constructed an element $U$ such that
\begin{enumerate}[(a)]
\item $U$ has finite propagation property;
\item $U$ is a generalized Fredholm operator;
\item $\ind(D_{\mathcal V}) = \ind(U)$.
\end{enumerate}
Recall that, to construct the index class of the element $U\in \mathcal B(\mathcal H_{C_r^\ast(\Gamma)})$, we choose a element $V$ such that  $ UV- 1$ and $VU-1$ are in $\mathcal K(\mathcal H_{C_r^\ast(\Gamma)})$  . Then the idempotent
\[ W \begin{pmatrix} 1 & 0 \\ 0 & 0\end{pmatrix} W^{-1} = \begin{pmatrix} UV + UV(1-UV) & (2 + UV)(1-UV) U \\ V(1-UV) & (1-UV)^2\end{pmatrix},\]
is a representative of the index class, where
\[ W = \begin{pmatrix} 1 & U \\ 0 & 1\end{pmatrix} \begin{pmatrix} 1 & 0 \\ -V & 1\end{pmatrix} \begin{pmatrix} 1 & U  \\ 0 & 1 \end{pmatrix}\begin{pmatrix} 0 & -1\\ 1 & 0 \end{pmatrix}. \]
However, an arbitrary choice of  $V$ cannot guarantee that the resulting idempotent $p$ still has finite propagation property. So to remedy this, we choose 
\[  V = U^\ast. \]  
Then clearly 
\[ p =  \begin{pmatrix} UU^\ast + UU^\ast(1-UU^\ast) & (2 + UU^\ast)(1-UU^\ast) U \\ U^\ast(1-UU^\ast) & (1-UU^\ast)^2\end{pmatrix} \]
has finite propagation property. In general, $UU^\ast -1$  and $U^\ast U - 1$ are not in $\mathcal K(\mathcal H_{C_r^\ast(\Gamma)})$. As a result, 
\[ p- \begin{pmatrix} 1 & 0 \\0 & 0\end{pmatrix}  \notin \mathcal K(\mathcal H_{C_r^\ast(\Gamma)}).\]
This is taken care of by the discussion in Section $\ref{subsec:k}$, since $p$ is $\tau$-close to $\mathcal K(\mathcal H_{C_r^\ast(\Gamma)})$ (in the sense of Definition $\ref{def:close}$)  when $\varphi$ is sufficiently close to $\chi$. Therefore,  $p$ defines a $K$-theory class in $K_0( \mathcal K(\mathcal H_{C_r^\ast(\Gamma)}))$. This class coincides with the index class $\ind(D_{\mathcal V})$. 

\begin{definition}\label{def:idemfp}
We call the idempotent $p$ constructed above  an \textit{idempotent of finite propagation} of the Dirac operator $D_{\mathcal V}$.
\end{definition}

\begin{remark}\label{remark:dim}
To deal with manifolds of dimension $n \neq 0 \pmod 8$ in the real case, we work with $\cl_n$-linear Dirac operators, cf. \cite[Section II.7]{BLMM89}. Here $\cl_{n}$ is the standard real Clifford algebra on $\mathbb R^n$ with $e_ie_j + e_je_i =  - 2 \delta_{ij}$. We recall the definition of $\cl_n$-linear Dirac operators in the following. Consider the standard representation $\ell$ of $\textup{Spin}_n$ on $\cl_n$ given by left multiplication. Let $P_{\textup{spin}}(X)$ be the principal $\textup{Spin}_n$-bundle of a $n$-dimensional spin manifold $X$, then we define $\mathfrak S$ to be the vector bundle
\[ \mathfrak S = P_{\textup{spin}}(X)\times_{\ell} \cl_n. \]
Now let $\mathcal V$ be a $C_r^\ast(\Gamma; \mathbb R)$-bundle over $X$ as before. We denote the associated Dirac operator on $\mathfrak S\otimes \mathcal V$ by 
\[ \mathfrak D: \vl^2(X, \mathfrak S\otimes \mathcal V) \to \vl^2(X, \mathfrak S\otimes \mathcal V).\]
Notice that the right multiplication of $\cl_n$ on $\mathfrak S$ commutes with $\ell$. So we see that $\mathfrak D$ in fact defines a  higher index class 
\[\ind(\mathfrak D)\in \widehat K_0(C_r^\ast(\Gamma; \mathbb R)\widehat\otimes \cl_n) \cong \widehat K_{n}(C_r^\ast(\Gamma; \mathbb R)) \cong K_n(C_r^\ast(\Gamma; \mathbb R)).\] Here $\widehat K_\ast$ stands for $\mathbb Z_2$-graded $K$-theory, $\widehat\otimes$ stands for $\mathbb Z_2$-graded tensor product, cf. \cite[Chapter III]{MK78}. Notice that for a trivially graded $C^\ast$-algebra $\mathcal A$, we have $ \widehat K_{n}(\mathcal A) \cong K_n(\mathcal A)$.

This approach works equally well for the complex case, in which case $K$-theory takes periodicity $2$ instead of $8$. 
\end{remark}

\begin{remark}
In fact, there is a more geometric approach in the complex case. With the same notation as before, we assume $X$ is an odd dimensional spin manifold. Then $\mathbb R\times X_\Gamma$ is a $(\mathbb Z\times \Gamma)$-cover of $\mathbb S^1\times X$, where $\mathbb S^1$ is the unit circle. Now define the corresponding $C_r^\ast(\mathbb Z\times \Gamma)$-bundle over $\mathbb S^1\times \Gamma$ by
\[  \mathcal W = (\mathbb R\times X_\Gamma)\times_{\mathbb Z\times \Gamma}C_r^\ast(\mathbb Z\times \Gamma). \]
Then we have $\ind(D_{\mathcal W}) \in K_0(C_r^\ast(\mathbb Z\times \Gamma)) = K_0(C_r^\ast(\Gamma)) \oplus K_1(C_r^\ast(\Gamma)).  $  In fact, $ \ind(D_{\mathcal W})$ lies in the second summand, that is, $ \ind(D_{\mathcal W})  \in   K_1(C_r^\ast(\Gamma)).$
\end{remark}

%

\section{A relative higher index theorem} \label{sec:rel}

In this section, we prove one of the main results, a \textit{relative higher index theorem},  of the paper . 

Let $X_0$ and $X_1$ be two even dimensional\footnote{In the real case, we assume $\dim X_0 = \dim X_1 \equiv 0 \pmod 8$.} spin manifolds with complete Riemannian metrics of positive scalar curvature towards infinity. Assume that we have compact subspaces $K_i\subset X_i$ such that there is an (orientation preserving) isometry $\diffeo: \Omega_0 \to \Omega_1$, where  $\Omega_i \subset X_i - K_i$ is a union of (not necessarily all) connected components of $X_i - K_i$  (see Figure $\ref{fig:ends}$). Let $\mathcal S_i$ be the corresponding spinor bundle over $X_i$. We assume that  $\diffeo$ lifts to a bundle isometry $\widetilde \diffeo: \mathcal S_0|_{\Omega_0} \to \mathcal S_1|_{\Omega_1}$.  

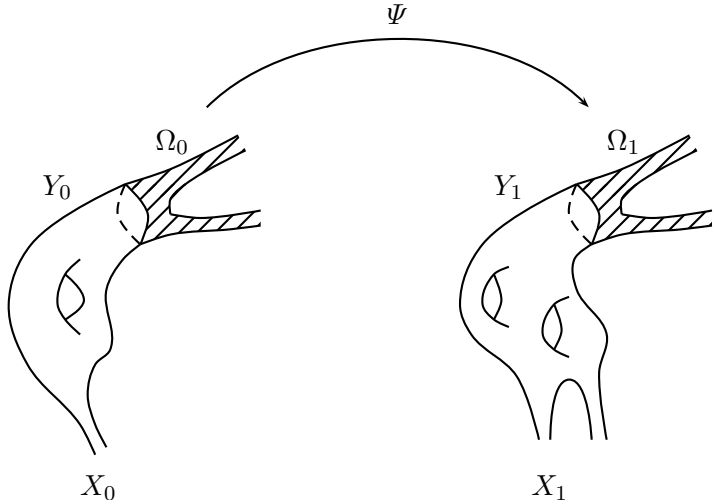
\begin{figure}[h]

\begin{pspicture*}(0, 0)(12, 7)
\rput[bl](3, 4.6){\rnode{A}{$\Omega_0$}}
\rput[bl](9, 4.6){\rnode{B}{$\Omega_1$}}
\psset{nodesep=.3, arcangle=45}
\ncarc{->}{A}{B}
\Aput{$\displaystyle \diffeo$}

\rput[bl](2, 0){$X_0$}
\rput[bl](8, 0){$X_1$}

\rput[bl](1.5, 4){$Y_0$}
\rput[bl](7.5, 4){$Y_1$}

\rput[bl](0, 0){ 

\pscurve(2.2,.6)(2, 1)(1.3, 1.8)(1.05, 2.6)(1.35, 3.4)(2.6, 4.2)
\pscurve(2.35, .7)(2.2, 1)(2.1, 1.4)(2.2, 1.8)(2.4, 2)(2.35, 2.8)(2.57, 3.2)(2.8, 3.4)


\pscurve(2, 2.2)(1.8, 2.4)(1.7, 2.6)(1.8, 3)(2,3.2)

\pscurve(1.8, 2.4) (2.05, 2.65)(1.8, 3)

\rput[bl](0,0){
\pscustom{  
\pscurve(2.8, 3.4)(3.2,3.53)(4, 3.6)(4.4, 3.65)
\gsave
\pscurve[liftpen=1](4.4, 3.85)(3.6, 3.75)(3.2, 3.8)
\fill[fillstyle=hlines,fillcolor=gray] \grestore}
\pscurve(4.4, 3.85)(3.6, 3.75)(3.2, 3.8)

\pscustom{  
\pscurve(2.8, 3.4)(2.9, 3.8)(2.6, 4.2) 
\gsave
\pscurve[liftpen=1](3.4, 4.2)(3.2, 4)(3.2, 3.8)
\fill[fillstyle=hlines,fillcolor=gray] \grestore}
\pscurve(3.4, 4.2)(3.2, 4)(3.2, 3.8)

\pscustom{  
\pscurve(2.6, 4.2)(3.2, 4.4)(4.1, 4.85)
\gsave
\pscurve[liftpen=1](4.2, 4.65)(3.4, 4.2)(3.2, 4)
\fill[fillstyle=hlines,fillcolor=gray] \grestore}

\pscurve(4.2, 4.65)(3.4, 4.2)(3.2, 4)
\pscurve[linestyle=dashed](2.8, 3.4)(2.5, 3.8)(2.6, 4.2)
}

}

\rput[bl](6,0){

\pscurve(2.1, .8)(1.8, 1.65)(1.2, 2.2)(1.05, 2.6)(1.35, 3.4)(2.6, 4.2)
\pscurve(3, .8)(2.9, 1.7)(3, 2.2)(2.6, 2.8)(2.57, 3.2)(2.8, 3.4)


\pscurve(2.25, .8)(2.5, 1.6)(2.8, .8) 
\pscurve(1.7, 2.3)(1.5, 2.4)(1.35, 2.6)(1.5, 3)(1.7, 3.1)
\pscurve(1.5, 2.4)(1.6, 2.65)(1.5, 3)

\pscurve(2.5, 1.9)(2.3, 2)(2.15, 2.2)(2.3, 2.6)(2.5, 2.7)
\pscurve(2.3, 2)(2.4, 2.25)(2.3, 2.6)
%

%
%
%
%

\rput[bl](0,0){
\pscustom{  
\pscurve(2.8, 3.4)(3.2,3.53)(4, 3.6)(4.4, 3.65)
\gsave
\pscurve[liftpen=1](4.4, 3.85)(3.6, 3.75)(3.2, 3.8)
\fill[fillstyle=hlines,fillcolor=gray] \grestore}
\pscurve(4.4, 3.85)(3.6, 3.75)(3.2, 3.8)

\pscustom{  
\pscurve(2.8, 3.4)(2.9, 3.8)(2.6, 4.2) 
\gsave
\pscurve[liftpen=1](3.4, 4.2)(3.2, 4)(3.2, 3.8)
\fill[fillstyle=hlines,fillcolor=gray] \grestore}
\pscurve(3.4, 4.2)(3.2, 4)(3.2, 3.8)

\pscustom{  
\pscurve(2.6, 4.2)(3.2, 4.4)(4.1, 4.85)
\gsave
\pscurve[liftpen=1](4.2, 4.65)(3.4, 4.2)(3.2, 4)
\fill[fillstyle=hlines,fillcolor=gray] \grestore}
\pscurve(4.2, 4.65)(3.4, 4.2)(3.2, 4)
\pscurve[linestyle=dashed](2.8, 3.4)(2.5, 3.8)(2.6, 4.2)
}

}

\end{pspicture*}

\caption{manifolds $X_0$ and  $X_1$ }\label{fig:ends}
\end{figure}


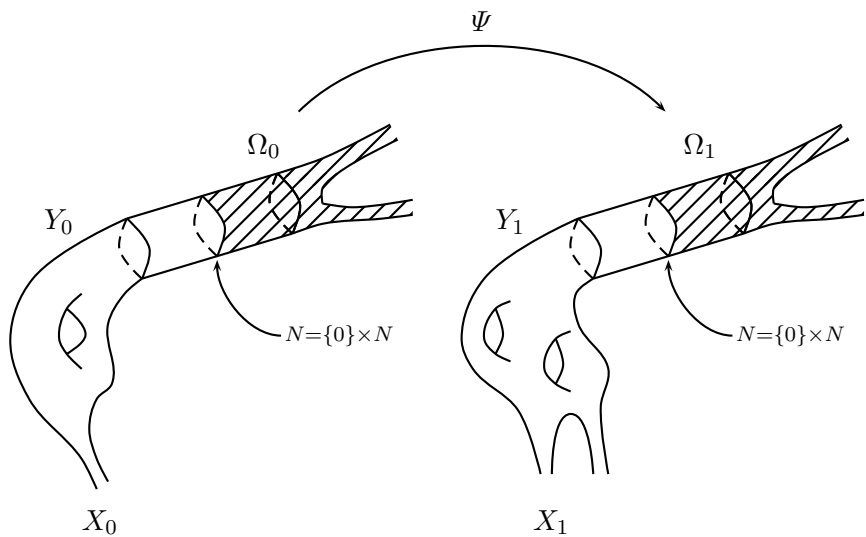
\begin{figure}[h]

\begin{pspicture*}(-1.5, 0)(11, 7)
\rput[bl](2.2, 5){\rnode{A}{$\Omega_0$}}
\rput[bl](8, 5){\rnode{B}{$\Omega_1$}}
\psset{nodesep=.3, arcangle=45}
\ncarc{->}{A}{B}
\Aput{$\displaystyle \diffeo$}


\rput[bl](-2, 0){

\pscurve(2.2,.6)(2, 1)(1.3, 1.8)(1.05, 2.6)(1.35, 3.4)(2.6, 4.2)
\pscurve(2.35, .7)(2.2, 1)(2.1, 1.4)(2.2, 1.8)(2.4, 2)(2.35, 2.8)(2.57, 3.2)(2.8, 3.4)


\pscurve(2, 2.2)(1.8, 2.4)(1.7, 2.6)(1.8, 3)(2,3.2)

\pscurve(1.8, 2.4) (2.05, 2.65)(1.8, 3)

\pscurve(2.8, 3.4)(2.9, 3.8)(2.6, 4.2) 
\pscurve[linestyle=dashed](2.8, 3.4)(2.5, 3.8)(2.6, 4.2)

\rput[bl](3.8, 3.7){\rnode{C}{}}
\rput[bl](4.7, 2.5){\rnode{D}{$\scriptstyle N = \{0\}\times N$}}

\psset{nodesep=.05}
\nccurve[angleA=-90, angleB=180]{<-}{C}{D}

\psline(2.6,4.2)(4.6,4.8)
\psline(2.8, 3.4)(4.8, 4)

\rput[bl](1, .3){

\pscurve[linestyle=dashed](2.8, 3.4)(2.5, 3.8)(2.6, 4.2)
\pscustom{
\pscurve(2.8, 3.4)(2.9, 3.8)(2.6, 4.2) 
\gsave
\pscurve[liftpen=1](3.6, 4.5) (3.9, 4.1)(3.8, 3.7)
\fill[fillstyle=hlines,fillcolor=gray] \grestore}
\pscurve(3.8, 3.7)(3.9, 4.1)(3.6, 4.5) 

\pscurve[linestyle=dashed](3.8, 3.7)(3.5, 4.1)(3.6, 4.5)
}

\rput[bl](2,.6){
\pscustom{  
\pscurve(2.8, 3.4)(3.2,3.53)(4, 3.6)(4.4, 3.65)
\gsave
\pscurve[liftpen=1](4.4, 3.85)(3.6, 3.75)(3.2, 3.8)
\fill[fillstyle=hlines,fillcolor=gray] \grestore}
\pscurve(4.4, 3.85)(3.6, 3.75)(3.2, 3.8)

\pscustom{  
\pscurve(2.8, 3.4)(2.9, 3.8)(2.6, 4.2) 
\gsave
\pscurve[liftpen=1](3.4, 4.2)(3.2, 4)(3.2, 3.8)
\fill[fillstyle=hlines,fillcolor=gray] \grestore}
\pscurve(3.4, 4.2)(3.2, 4)(3.2, 3.8)

\pscustom{  
\pscurve(2.6, 4.2)(3.2, 4.4)(4.1, 4.85)
\gsave
\pscurve[liftpen=1](4.2, 4.65)(3.4, 4.2)(3.2, 4)
\fill[fillstyle=hlines,fillcolor=gray] \grestore}
\pscurve(4.2, 4.65)(3.4, 4.2)(3.2, 4)
\pscurve[linestyle=dashed](2.8, 3.4)(2.5, 3.8)(2.6, 4.2)
}

 }


\rput[bl](4,0){

\pscurve(2.1, .8)(1.8, 1.65)(1.2, 2.2)(1.05, 2.6)(1.35, 3.4)(2.6, 4.2)
\pscurve(3, .8)(2.9, 1.7)(3, 2.2)(2.6, 2.8)(2.57, 3.2)(2.8, 3.4)


\pscurve(2.25, .8)(2.5, 1.6)(2.8, .8) 
\pscurve(1.7, 2.3)(1.5, 2.4)(1.35, 2.6)(1.5, 3)(1.7, 3.1)
\pscurve(1.5, 2.4)(1.6, 2.65)(1.5, 3)

\pscurve(2.5, 1.9)(2.3, 2)(2.15, 2.2)(2.3, 2.6)(2.5, 2.7)
\pscurve(2.3, 2)(2.4, 2.25)(2.3, 2.6)

\pscurve(2.8, 3.4)(2.9, 3.8)(2.6, 4.2) 
\pscurve[linestyle=dashed](2.8, 3.4)(2.5, 3.8)(2.6, 4.2)

\rput[bl](3.8, 3.7){\rnode{E}{}}
\rput[bl](4.7, 2.5){\rnode{F}{$\scriptstyle N = \{0\}\times N$}}

\psset{nodesep=.05}
\nccurve[angleA=-90, angleB=180]{<-}{E}{F}

\psline(2.6,4.2)(4.6,4.8)
\psline(2.8, 3.4)(4.8, 4)

\rput[bl](1, .3){

\pscurve[linestyle=dashed](2.8, 3.4)(2.5, 3.8)(2.6, 4.2)
\pscustom{
\pscurve(2.8, 3.4)(2.9, 3.8)(2.6, 4.2) 
\gsave
\pscurve[liftpen=1](3.6, 4.5) (3.9, 4.1)(3.8, 3.7)
\fill[fillstyle=hlines,fillcolor=gray] \grestore}
\pscurve(3.8, 3.7)(3.9, 4.1)(3.6, 4.5) 

\pscurve[linestyle=dashed](3.8, 3.7)(3.5, 4.1)(3.6, 4.5)
}

\rput[bl](2,.6){
\pscustom{  
\pscurve(2.8, 3.4)(3.2,3.53)(4, 3.6)(4.4, 3.65)
\gsave
\pscurve[liftpen=1](4.4, 3.85)(3.6, 3.75)(3.2, 3.8)
\fill[fillstyle=hlines,fillcolor=gray] \grestore}
\pscurve(4.4, 3.85)(3.6, 3.75)(3.2, 3.8)

\pscustom{  
\pscurve(2.8, 3.4)(2.9, 3.8)(2.6, 4.2) 
\gsave
\pscurve[liftpen=1](3.4, 4.2)(3.2, 4)(3.2, 3.8)
\fill[fillstyle=hlines,fillcolor=gray] \grestore}
\pscurve(3.4, 4.2)(3.2, 4)(3.2, 3.8)

\pscustom{  
\pscurve(2.6, 4.2)(3.2, 4.4)(4.1, 4.85)
\gsave
\pscurve[liftpen=1](4.2, 4.65)(3.4, 4.2)(3.2, 4)
\fill[fillstyle=hlines,fillcolor=gray] \grestore}
\pscurve(4.2, 4.65)(3.4, 4.2)(3.2, 4)
\pscurve[linestyle=dashed](2.8, 3.4)(2.5, 3.8)(2.6, 4.2)
}

}
\rput[bl](0, 0){$X_0$}
\rput[bl](6, 0){$X_1$}

\rput[bl](-.5, 4){$Y_0$}
\rput[bl](5.5, 4){$Y_1$}

\end{pspicture*}

\caption{stretched manifolds $X_0$ and $X_1$}\label{fig:surfaces}

\end{figure}

Let $(X_i)_\Gamma$ be a $\Gamma$-cover of $X_i$, where $\Gamma$ is a discrete group. Denote by 
\[  \pi_i : (X_i)_\Gamma \to X_i  \] 
the corresponding covering map. We assume that $\diffeo$ lifts to an isometry of the covers, also denoted by  $\widetilde \diffeo$, i.e. the following diagram commutes.
\[  \xymatrix{ \pi_0^{-1}(\Omega_0)  \ar[r]^{\widetilde \diffeo} \ar[d]_{\pi_0}  & \pi_1^{-1}(\Omega_1) \ar[d]^{\pi_1} \\
   \Omega_0 \ar[r]^{\diffeo}  & \Omega_1
}
\]
Define a flat $C_r^\ast(\Gamma)$-bundle $\mathcal V_i$ on $X_i$ by 
\[  \mathcal V_i = (X_i)_\Gamma\times_\Gamma C_r^\ast(\Gamma). \]
Then $\widetilde \diffeo$ induces a bundle isometry from $\mathcal V_0 |_{\Omega_0} $ to $\mathcal V_1|_{\Omega_1} $, which we still denote by $\widetilde\diffeo$. Notice that $\widetilde \diffeo:  \mathcal V_0 |_{\Omega_0}\to \mathcal V_0 |_{\Omega_1}$ covers the isometry $\diffeo: \Omega_0 \to \Omega_1$, that is, the following diagram commutes.
\[  \xymatrix{ \mathcal V_0|_{\Omega_0}  \ar[r]^{\widetilde \diffeo} \ar[d]  & \mathcal V_1|_{\Omega_1} \ar[d] \\
   \Omega_0 \ar[r]^{\diffeo}  & \Omega_1
}
\]
Let $D_i = D_{\mathcal V_i}$ be the associated Dirac operator on $X_i$, $ i =0, 1$. Then we have 
\[  D_1 = \widetilde \diffeo \circ D_0 \circ \widetilde \diffeo^{-1}  \]
on $\Omega_1$. We say $D_0$ and $D_1$ agree on $\Omega = \Omega_0 \cong \Omega_1.$

Let $N$ be a compact hypersurface in $\Omega\cong \Omega_i$ such that $N$ cuts $X_i$ into two components. We separate off the component that is inside $\Omega_i$ and denote the remaining part of $X_i$ by $Y_i$ (see Figure $\ref{fig:ends}$). Note that a deformation of the metric in a compact subset of a manifold does not affect the scalar curvature towards infinity, neither does it change the $K$-theory class of the higher index of the associated Dirac operator. So without loss of generality,  we can assume $Y_i$ has product metric in a small neighborhood of $N$. 

In fact, in order to make use of the finite propagation property of our higher index classes, we further deform the metric near a collar neighborhood $(-\delta, \delta)\times N$ of $N$ so that $ (-\delta, \delta)\times N$ becomes $(-\ell, \ell)\times N$ for $\ell$ sufficiently large. Here we assume the standard Euclidean metric along the interval $(-\ell, \ell)$. Now $ N = \{0\}\times N$ cuts $X_i$ into two components. We separate off the component that is inside $\Omega_i$ and denote the remaining part of $X_i$ by $Y_i\cup ( (-\ell, 0]\times N) $ (see Figure $\ref{fig:surfaces}$).

\begin{remark}
 Recall that our choice of the normalizing function $\chi$ (Section $\ref{sec:hic}$)  depends on the lower bound of $D_{\mathcal V} + \varepsilon \rho$. We claim that the stretching performed above  does not affect the choice of $\chi$. Indeed, since the stretching does not change the scalar curvature on the cylindrical neighborhood of $N$, it follows from the proof of Lemma $\ref{lemma:lb}$ that the same lower bound works for the operator $D_{\mathcal V} + \varepsilon \rho$ before and after the stretching.  
\end{remark}

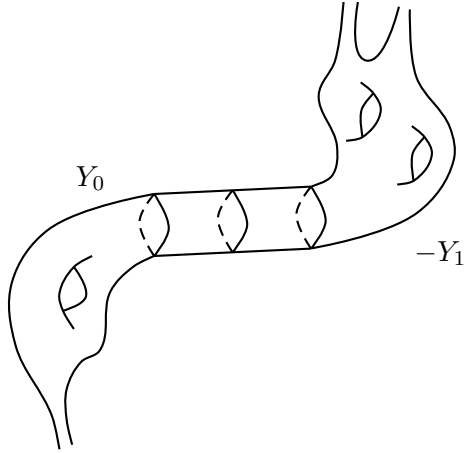
\begin{figure}[h]

\begin{pspicture*}(0,4)(10, 10)

\rput[bl]{-14}(1,4){

\pscurve(2.2,.6)(2, 1)(1.3, 1.8)(1.05, 2.6)(1.35, 3.4)(2.6, 4.2)
\pscurve(2.35, .7)(2.2, 1)(2.1, 1.4)(2.2, 1.8)(2.4, 2)(2.35, 2.8)(2.57, 3.2)(2.8, 3.4)


\pscurve(2, 2.2)(1.8, 2.4)(1.7, 2.6)(1.8, 3)(2,3.2)

\pscurve(1.8, 2.4) (2.05, 2.65)(1.8, 3)

\pscurve(2.8, 3.4)(2.9, 3.8)(2.6, 4.2) 
\pscurve[linestyle=dashed](2.8, 3.4)(2.5, 3.8)(2.6, 4.2)

\psline(2.6,4.2)(4.6,4.8)
\psline(2.8, 3.4)(4.8, 4)

\rput[bl](1, .3){

\pscurve[linestyle=dashed](2.8, 3.4)(2.5, 3.8)(2.6, 4.2)
\pscurve(2.8, 3.4)(2.9, 3.8)(2.6, 4.2) 
\pscurve(3.8, 3.7)(3.9, 4.1)(3.6, 4.5) 

\pscurve[linestyle=dashed](3.8, 3.7)(3.5, 4.1)(3.6, 4.5)
}

 }

\rput[bl](3.5, 7.5){$Y_0$}

\rput[bl](8, 6.5){$-Y_1$}

\rput[bl]{166}(10.17, 10.17){

\pscurve(2.1, .8)(1.8, 1.65)(1.2, 2.2)(1.05, 2.6)(1.35, 3.4)(2.6, 4.2)
\pscurve(3, .8)(2.9, 1.7)(3, 2.2)(2.6, 2.8)(2.57, 3.2)(2.8, 3.4)


\pscurve(2.25, .8)(2.5, 1.6)(2.8, .8) 
\pscurve(1.7, 2.3)(1.5, 2.4)(1.35, 2.6)(1.5, 3)(1.7, 3.1)
\pscurve(1.5, 2.4)(1.6, 2.65)(1.5, 3)

\pscurve(2.5, 1.9)(2.3, 2)(2.15, 2.2)(2.3, 2.6)(2.5, 2.7)
\pscurve(2.3, 2)(2.4, 2.25)(2.3, 2.6)

}

\end{pspicture*}
\caption{manifold $X_2$}\label{fig:glue}
\end{figure}

Now we can glue $Y_0\cup((-\ell, 0]\times N)$ and $Y_1\cup( (-\ell, 0]\times N)$ along $N = \{0\}\times N$. To do this, we need to reverse the orientation of one of the manifolds, say, $Y_1\cup((-\ell, 0]\times N)$.  We denote the resulting manifold from this gluing by  $X_2$ (see Figure $\ref{fig:glue}$). The spinor bundles $\mathcal S_0$ over $Y_0\cup((-\ell, 0]\times N)$ and $\mathcal S_1$ over $Y_1\cup((-\ell, 0]\times N)$ are glued together by the Clifford multiplication $c(v)$ to give a spinor bundle over $X_2$, where $v =\frac{d}{du}$ is the inward unit normal vector near the boundary of $Y_0\cup((-\ell, 0]\times N)$. Moreover, the bundles $\mathcal V_0|_{Y_0}$ and $\mathcal V_1|_{Y_1}$ are glued together by $\widetilde\diffeo$ (near the boundary) to define a flat bundle $\mathcal V_2$ on $X_2$. Let $D_2 = D_{\mathcal V_2}$ be the associated Dirac operator on $X_2$.  

Similarly,  we can use two copies of $Y_1$ to construct a double of $Y_1$. We define the manifold (see Figure $\ref{fig:double}$)
\[ X_3 = (Y_1\cup((-\ell, 0]\times N) ) \bigcup_{\{0\}\times N} -(Y_1\cup( (-\ell, 0]\times N ) )\]  and denote its associated Dirac operator by $D_3 = D_{\mathcal V_3}$.

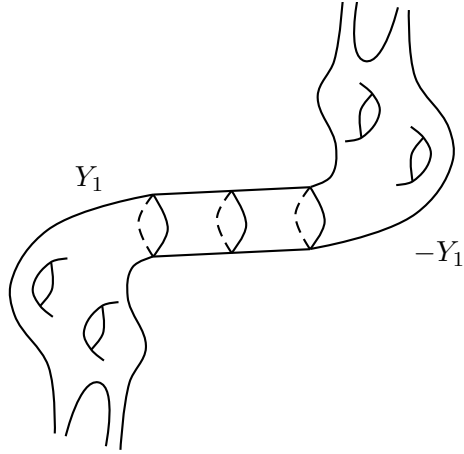
\begin{figure}[h]

\begin{pspicture*}(0,4)(10, 10)

\rput[bl]{-14}(1,4){

\pscurve(2.1, .8)(1.8, 1.65)(1.2, 2.2)(1.05, 2.6)(1.35, 3.4)(2.6, 4.2)
\pscurve(3, .8)(2.9, 1.7)(3, 2.2)(2.6, 2.8)(2.57, 3.2)(2.8, 3.4)


\pscurve(2.25, .8)(2.5, 1.6)(2.8, .8) 
\pscurve(1.7, 2.3)(1.5, 2.4)(1.35, 2.6)(1.5, 3)(1.7, 3.1)
\pscurve(1.5, 2.4)(1.6, 2.65)(1.5, 3)

\pscurve(2.5, 1.9)(2.3, 2)(2.15, 2.2)(2.3, 2.6)(2.5, 2.7)
\pscurve(2.3, 2)(2.4, 2.25)(2.3, 2.6)

\pscurve(2.8, 3.4)(2.9, 3.8)(2.6, 4.2) 
\pscurve[linestyle=dashed](2.8, 3.4)(2.5, 3.8)(2.6, 4.2)

\psline(2.6,4.2)(4.6,4.8)
\psline(2.8, 3.4)(4.8, 4)

\rput[bl](1, .3){

\pscurve[linestyle=dashed](2.8, 3.4)(2.5, 3.8)(2.6, 4.2)
\pscurve(2.8, 3.4)(2.9, 3.8)(2.6, 4.2) 
\pscurve(3.8, 3.7)(3.9, 4.1)(3.6, 4.5) 

\pscurve[linestyle=dashed](3.8, 3.7)(3.5, 4.1)(3.6, 4.5)
}

 }

\rput[bl](3.5, 7.5){$Y_1$}

\rput[bl](8, 6.5){$-Y_1$}

\rput[bl]{166}(10.17, 10.17){

\pscurve(2.1, .8)(1.8, 1.65)(1.2, 2.2)(1.05, 2.6)(1.35, 3.4)(2.6, 4.2)
\pscurve(3, .8)(2.9, 1.7)(3, 2.2)(2.6, 2.8)(2.57, 3.2)(2.8, 3.4)


\pscurve(2.25, .8)(2.5, 1.6)(2.8, .8) 
\pscurve(1.7, 2.3)(1.5, 2.4)(1.35, 2.6)(1.5, 3)(1.7, 3.1)
\pscurve(1.5, 2.4)(1.6, 2.65)(1.5, 3)

\pscurve(2.5, 1.9)(2.3, 2)(2.15, 2.2)(2.3, 2.6)(2.5, 2.7)
\pscurve(2.3, 2)(2.4, 2.25)(2.3, 2.6)

}

\end{pspicture*}
\caption{manifold $X_3$}\label{fig:double}
\end{figure}


We have the following relative higher index theorem. 

\begin{theorem}\label{thm:rit}
\[ \ind(D_2)  = \ind(D_0) - \ind(D_1) \]
in $K_0(C_r^\ast(\Gamma))$.
\end{theorem}

\begin{remark}
If $\dim X = n$, then we have 
\[ \ind(D_2)  = \ind(D_0) - \ind(D_1) \]
in $K_n(C_r^\ast(\Gamma))$ (resp. $K_n(C_r^\ast(\Gamma; \mathbb R))$ in the real case), cf. Remark $\ref{remark:dim}$.

\end{remark}

Before we prove the theorem, let us fix some notation. Let $p$ be an idempotent of finite propagation (in the sense of Definition $\ref{def:idemfp}$) for $D_0$ and  $q$ an idempotent of finite propagation for $D_1$. Since $p$ and $q$ have finite propagation property and the cylinder $(-\ell , \ell)\times N$ is sufficiently long (that is, $\ell$ is sufficiently large), we have  
\[ p(\sigma)  = u^{\ast} q u (\sigma)\]
 for all $\sigma\in \vl^2(X_0\backslash (Y_0\cup ((-\ell, 0]\times N) ), \mathcal S_0\otimes \mathcal V_0)$, where 
\[u: \vl^2(X_0 \backslash (Y_0\cup ( (-\ell, 0] \times N) ), \mathcal S_0\otimes \mathcal V_0) \to \vl^2(X_1 \backslash (Y_1\cup ( (-\ell, 0]\times N ) ), \mathcal S_1\otimes \mathcal V_1)\] is the unitary operator induced by the isometry 
$\widetilde \diffeo: \mathcal S_0\otimes \mathcal V_0|_{\Omega_0} \to \mathcal S_0\otimes \mathcal V_1|_{\Omega_1}.$

\begin{definition}
Define the following Hilbert modules over $C_r^\ast(\Gamma)$:
\[ \mathcal H_1 = \vl^2(Y_0\cup ((-\ell, 0]\times N), \mathcal S_0\otimes \mathcal V_0), \]
\[\mathcal H_2 = \vl^2( [0, \ell]\times N, \mathcal S_0\otimes \mathcal V_0), \]
\[ \mathcal H_3 =  \vl^2(X_0 \backslash (Y_0\cup ((-\ell, \ell]\times N ) ), \mathcal S_0\otimes \mathcal V_0), \]
\[ \mathcal H_4 =  \vl^2(Y_1\cup( (-\ell, 0]\times N), \mathcal S_1\otimes \mathcal V_1). \]
\end{definition}
Notice that 
\[ \vl^2(X_0, \mathcal S_0\times \mathcal V_0) = \mathcal H_1 \oplus \mathcal H_2 \oplus \mathcal H_3 \]
and 
\[ \vl^2(X_1, \mathcal S_1\times \mathcal V_1) = u(\mathcal H_2) \oplus u(\mathcal H_3)\oplus \mathcal H_4. \]
Let us denote 
\[\widetilde {\mathcal H}_{C_r^\ast(\Gamma)} = \mathcal H_1\oplus \mathcal H_2 \oplus \mathcal H_3 \oplus \mathcal H_4.\]
By finite propagation property of $p$ and $q$, we have 
\[ p = \begin{pmatrix} p_{11} & p_{12} & 0  & 0 \\  p_{21} & p_{22} & p_{23}& 0 \\ 0 & P_{32} & p_{33}  & 0  \\ 0 & 0 & 0 & 0 \end{pmatrix}  \quad \textup{and} \quad  q=  \begin{pmatrix} 0  & 0 & 0 & 0\\ 0 & q_{22} & q_{23} & q_{24} \\ 0 & q_{32} & q_{33}  & 0 \\ 0 & q_{42} & 0  & q_{44}  \end{pmatrix}\]
in $  \mathcal B(\widetilde {\mathcal H}_{C_r^\ast(\Gamma)} )$.  Here $  \mathcal B(\widetilde {\mathcal H}_{C_r^\ast(\Gamma)} )$ is the space of all adjointable operators in the Hilbert module $\widetilde {\mathcal H}_{C_r^\ast(\Gamma)}$.

\begin{proof}[\textbf{Proof of Theorem }$\ref{thm:rit}$]
Denote 
\[ e_1 = \begin{pmatrix} 1 & 0 \\ 0 & 0 \end{pmatrix} \in \mathcal B(\vh_1) \quad \textup{and} \quad  e_4 = \begin{pmatrix} 1 & 0 \\ 0 & 0 \end{pmatrix} \in \mathcal B(\vh_4), \]
where the matrix form is used to denote the $\mathbb Z_2$-grading of $\vh_2$ and $\vh_3$. 
Define 
\[ \tilde p = \begin{pmatrix} p_{11} & p_{12} & 0  & 0 \\  p_{21} & p_{22} & p_{23}& 0 \\ 0 & P_{32} & p_{33}  & 0  \\ 0 & 0 & 0 & e_4 \end{pmatrix}  \quad \textup{and} \quad  \tilde q=  \begin{pmatrix} e_1  & 0 & 0 & 0\\ 0 & q_{22} & q_{23} & q_{24} \\ 0 & q_{32} & q_{33}  & 0 \\ 0 & q_{42} & 0  & q_{44}  \end{pmatrix}\]
Notice that $\tilde p - \tilde q $ is $\tau$-close  to $ \mathcal K(\widetilde {\mathcal H}_{C_r^\ast(\Gamma)})$ (in the sense of Definition $\ref{def:close}$). By applying the difference construction (cf. Section $\ref{sec:diff}$) to $(\tilde p, \tilde q)$, we obtain
\[  \ind(D_0) - \ind(D_1) =  E(\tilde p, \tilde q)  \in K_0(\mathcal K(\widetilde {\mathcal H}_{C_r^\ast(\Gamma)}) ) \cong K_0(C_r^\ast(\Gamma)).\]
We point out that,  due to the presence of the term  $\tilde p - \tilde q$ in all the nonzero entries in $E_0(\tilde p, \tilde q)$ (cf. Formula $\eqref{eq:diff}$ in Section $\ref{sec:diff}$), a straightforward calculation shows that the entries $p_{23}, p_{32} $ and $p_{33}$ (resp. $q_{23}, q_{32} $ and $q_{33}$) in the matrix $\tilde p$  (resp. $\tilde q$) do not appear in $E_0(\tilde p, \tilde q)$. In other words, the summand $\mathcal H_3$ ``disappears'' when we pass to $E_0(\tilde p, \tilde q)$.

Let $p_1$ (resp. $q_1$) be an idempotent of finite propagation for $D_2$ (resp. $D_3$). Similarly, define $\widetilde p_1$ and $\widetilde q_1$ as above, but in the Hilbert module 
\[ \widetilde \vh'_{C^\ast_r(\Gamma)} = \mathcal H_1\oplus \mathcal H_2 \oplus \mathcal H'_3 \oplus \mathcal H_4 \]
where $ \mathcal H'_3 =  \vl^2(-Y_1, \mathcal S_1\otimes \mathcal V_1) $.  Note that 
\[ \vl^2(X_2, \mathcal S_0\times \mathcal V_0) = \mathcal H_1 \oplus \mathcal H_2 \oplus \mathcal H'_3 \textup{ and } \vl^2(X_3, \mathcal S_0\times \mathcal V_0) = \mathcal H_2 \oplus \mathcal H'_3 \oplus \mathcal H_4. \]
Then the difference construction gives 
\[  \ind(D_2) - \ind(D_3) =  E(\tilde p_1, \tilde q_1)  \in K_0(\mathcal K(\widetilde {\mathcal H}_{C_r^\ast(\Gamma)}) ) \cong K_0(C_r^\ast(\Gamma)).\]
Similarly, we see that the summand $\mathcal H'_3$ ``disappears'' when we pass to $E_0(\tilde p_1, \tilde q_1)$. In fact, we have $E_0(\tilde p_1, \tilde q_1)  =  E_0(\tilde p, \tilde q)$
as matrices of operators in  $\mathcal B(\vh_1 \oplus \vh_2 \oplus \vh_4)$. Therefore, we have 
\[ \ind(D_2) - \ind(D_3)  = \ind(D_0) - \ind(D_1). \]
Now since $D_3$ is the associated Dirac operator over a double, it follows from  Theorem $\ref{thm:double}$ below that $\ind(D_3) = 0$. This finishes the proof.  

\end{proof}

\section{Invertible Doubles}\label{sec:id}
In this section, we carry out an invertible double construction (Theorem $\ref{thm:double}$) for manifolds with (real or complex) $C^\ast$-vector bundles. This generalizes the invertible double construction for manifolds with classical vector bundles (i.e. $\mathbb C$-vector bundles or $\mathbb R$-vector bundles), cf. \cite[Chapter 9]{BBKW93}. For simplicity, we only state and prove the results for the complex case. The real case is proved by exactly the same argument. 

 Let $Y_1$ be an even dimensional complete manifold with boundary $N$, where $N$ is a closed manifold. Assume that the Riemannian metric on $Y_1$ has positive scalar curvature towards infinity. Denote by $\mathcal S_1$ a Clifford bundle over $Y_1$. Let $\mathcal V$ be a $\mathcal A$-bundle over $Y_1$, where $\mathcal A$ is a $C^\ast$-algebra. Assume all metrics have product structures near the boundary. We denote a copy of $Y_1$ with the reversed orientation by $Y_2 = -Y_1$ and denote the corresponding Clifford bundle by $\mathcal S_2$. We glue $Y_1$ and $Y_2$ along a tubular neighborhood of the boundary to obtain a double $\widetilde Y$ of $Y_1$. Now the bundles $\mathcal S_1\otimes \mathcal V$ and $\mathcal S_2\otimes \mathcal V$ are glued together by the Clifford multiplication $c(v)$, where $v =\frac{d}{du}$ is the inward unit normal vector near the boundary of $Y_1$. We denote the resulting bundle on $\widetilde Y$ by $\widetilde {\mathcal S}\otimes \widetilde {\mathcal V}$. Note that 
\[  \widetilde {\mathcal S}^{\pm}  = \mathcal S_1^{\pm} \cup_{c(v)} \mathcal S_2^{\mp}. \]
In particular, a section of $\widetilde S^+\otimes \widetilde {\mathcal V}$ can be identified with a pair $(s_1, s_2)$ such that $s_1$ is a section of $\mathcal S_1^+\otimes \mathcal V$, $s_2$ is a section of $S_2^-\otimes \mathcal V $ and near the boundary 
\[  s_2 = c(v) s_1. \]
Denote the Dirac operator over $Y_i$ by
\[ D_i^\pm : \Gamma(Y_i, \mathcal S_i^\pm \otimes \mathcal V) \to \Gamma(Y_i, \mathcal S_i^\mp\otimes \mathcal V). \] Then the Dirac operator $\widetilde D$ on $\widetilde Y$ is identified with 
\[  \widetilde D^\pm (s_1, s_2) =  (D^\pm_1 s_1, D^\mp_2 s_2). \]

\begin{theorem}\label{thm:double}
The operator $\widetilde D$ is bounded below, i.e., there exists a constant $C$ such that 
\[  \|\sigma\| \leq C \|\widetilde D \sigma\| \]
for all $\sigma \in \Gamma^\infty(\widetilde Y, \widetilde {\mathcal S}\otimes \mathcal V)$. In particular, the higher index class $\ind(\widetilde D)$ is zero. 	
\end{theorem}
\begin{proof}
Since each $\sigma \in \Gamma^\infty(\widetilde Y, \widetilde {\mathcal S}^+\otimes \widetilde {\mathcal V)}$ can be identified with a pair $(\sigma_1, \sigma_2)$ such that $\sigma_1$ is a section of $S_1^+\otimes \mathcal V$, $\sigma_2$ is a section of $S_2^-\otimes \mathcal V $ and near the boundary 
\[  \sigma_2 = c(v) \sigma_1. \]
Therefore, by divergence theorem, we have 
\begin{align}
&\int_{Y_1} \langle D^+\sigma_1, \sigma_2\rangle - \int_{Y_1} \langle \sigma_1, D^-\sigma_2 \rangle = - \int_{\partial Y_1} \langle c(v) \sigma_1 , \sigma_2 \rangle = \int_{N} \langle \sigma_1, \sigma_1\rangle.\label{eq:div}
\end{align}
It follows that there exists $k_1 >0$ such that 
\begin{align*}
&\|\sigma|_{N} \|^2 = \int_{N} \langle \sigma_1, \sigma_1\rangle \leq k_1 \|\widetilde D^+ \sigma\| \|\sigma\| 
\end{align*}
Similarly, we have $\|\sigma|_N\|^2 \leq    k_2 \| \widetilde D^- \sigma\|  \|\sigma\|$
for all $\sigma \in \Gamma^\infty(\widetilde Y , \widetilde{\mathcal S}^-\otimes \widetilde {\mathcal V})$.  Therefore, there exist a constant $K_0$ such that 
\begin{equation}\label{eq:divineq}
  \|\sigma|_N\|^2 \leq  K_0 \| \widetilde D \sigma\|\|\sigma\| 
\end{equation}
for all $\sigma \in \Gamma^\infty(\widetilde Y , \widetilde {\mathcal S}\otimes \widetilde{\mathcal V})$.

Let $\Omega = (-\delta, \delta)\times N$ be a small tubular neighborhood of $N$ in $\widetilde Y$. Denote by $N_u = \{u\}\times N$ for $u\in (\delta, \delta)$. On the cylinder $ (-\delta, \delta)\times N$, we have 
\[  \widetilde D = c(u)(\frac{d}{du} + A) \]
where $c(u)$ is the Clifford multiplication of the normal direction $\frac{d}{du}$ and $A$ is the Dirac operator on $N$. So we have a situation which is a special case of Lemma $\ref{lemma:ineq}$ below. It follows immediately from Lemma $\ref{lemma:ineq}$ that there exists a constant $K_1$ such that 
\[  \|\sigma|_{\Omega}\|^2 \leq K_1\left(\|(\widetilde D\sigma)|_\Omega\|^2 + \|\sigma|_N\|^2 \right)\]
for all $\sigma\in \in \Gamma^\infty(\widetilde Y , \widetilde {\mathcal S}\otimes \widetilde{\mathcal V})$. Combined with the inequality $\eqref{eq:divineq}$, this implies that 
\[ \|\sigma|_{\Omega}\|^2 \leq K_2 (\|\widetilde D\sigma\|^2 +  \| \widetilde D \sigma\|\|\sigma\| ).\]
Now by the technical estimate (or rather its corollary $\ref{cor:sn}$) below, we have $  \|\sigma\| \leq C_1 \|\sigma|_\Omega\| + C_2\|\widetilde D\sigma\|$,  or equivalently, 
\[  \|\sigma\|^2 \leq C'_1 \|\sigma|_\Omega\|^2 + C'_2\|\widetilde D\sigma\|^2.\]
Therefore there exists a constant $C_0 >0$ such that  
\[  \|\sigma\|^2 \leq C_0(\|\widetilde D \sigma\|\|\sigma\| + \|\widetilde D\sigma\|^2) \]
i.e. 
\[   \frac{\|\widetilde D \sigma\|}{\|\sigma\|} + \left(\frac{\|\widetilde D\sigma\|}{\|\sigma\|}\right)^2 \geq \frac{1}{C_0} \]
for all nonzero $\sigma \in \Gamma^\infty(\widetilde Y , S\otimes \mathcal V)$. So $\inf_{\sigma \neq 0} \frac{\|\widetilde D \sigma\|}{\|\sigma\|} > 0$.  This finishes the proof.

\end{proof}

\begin{remark}
The above proof works for all dimensions, with obvious modifications by using $\cl_n$-linear Dirac operators as in Remark $\ref{remark:dim}$. 
\end{remark}

\begin{remark}
We point out that there is a natural (orientation reversing) reflection on the double $\widetilde Y$.  The reflection commutes with the Dirac operator $\widetilde D$ on $\widetilde Y$. Using this, one sees that twice of the higher index class of $\widetilde D$ is zero, that is, $ 2\,\ind(\widetilde D) = 0 $ (for both the real and the complex cases).  
\end{remark}

\begin{remark}
We emphasize that the above proof works for both the real and the complex cases. 
\end{remark}

\begin{remark}
In the \textit{complex} case, there is in fact a  simpler way to show that $\ind(\widetilde D) =0$.  We thank Ulrich Bunke for pointing this out to us. We provide the argument in the following. Note that however this argument does not work in the real case. Let $\widetilde Y$ be as above (of even dimension). Denote the grading operator on $\widetilde {\mathcal S}$ by $\varepsilon$ and the reflection on $\widetilde Y$ (and its induced action on $\widetilde{\mathcal S}\otimes \widetilde{\mathcal V}$) by $J$. Define $ E = i  J \varepsilon$. Notice that $J$ anticommutes with $\varepsilon$ and commutes with $\widetilde D$. So $E^2 =1 $ and $E\widetilde D + \widetilde D E =0$.  Then, since $ (\widetilde D + tE)^2 = \widetilde D^2 + t^2$ is invertible when $t\in (0, 1]$, we see that $\widetilde D + tE$ is invertible for all $t\in(0, 1]$. Therefore, $\ind (\widetilde D) = \ind(\widetilde D + tE) = 0$ by homotopy invariance of the index map. 
\end{remark}

\subsection{A technical theorem}
In this subsection, we prove the technical estimate that was used in the proof of Theorem $\ref{thm:double}$. 

First let us consider the case of compact manifolds. Let $X$ be a compact Riemannian manifold and $\mathcal S$ a $\cl(TX)$-bundle with $\cl(TX)$-compatible connection. Let $\mathcal A$ be a $C^\ast$-algebra and let $\mathcal V$ be a $\mathcal A$-bundle over $X$. Denote by $D$ the associated generalized Dirac operator 
\[ D: \Gamma^\infty(X; \mathcal S\otimes \mathcal V) \to \Gamma^\infty(X;  \mathcal S\otimes \mathcal V). \]
Denote by $d(\cdot, \cdot)$ the Riemannian distance on $X$.  Then for $\lambda >0$, we define
 \[ \Omega^\lambda = \{ x\in X \mid d(x, \Omega) < \lambda \}\]
for  any open subset $\Omega$ of $X$. In the following, $\langle \, ,  \rangle $ stands for the $\mathcal A$-valued Hermitian product on $\mathcal V$ and  $\| \cdot \|$ denotes the $L^2$-norm on $\vl^2(X; \mathcal S\otimes \mathcal V)$, unless otherwise specified. 

\begin{theorem}\label{thm:sn}
With the above notation, fix an open subset $\Omega$ of $X$. Then there are constants $C_1$ and $C_2$ such that 
\[  \|\sigma\| \leq  C_1 \|\sigma|_\Omega \| + C_2\|D\sigma\| \] 
for all $\sigma\in \Gamma^\infty(X; \mathcal S\otimes \mathcal V)$. Here $\sigma|_\Omega$ is the restriction of $\sigma$ to $\Omega$. 
\end{theorem}
\begin{proof}
We reduce the theorem to the following claim.
\begin{claim}\label{claim}
There exists a constant $\lambda> 0$ such that for any open subset $\Omega\subset X$, there exist constants $K_{\Omega, 1}$ and $ K_{\Omega, 2}$ such that 
\[  \|\sigma|_{\Omega^\lambda}\| \leq  K_{\Omega, 1}\|\sigma|_{\Omega}\| + K_{\Omega, 2} \|D\sigma\|  \]
for all $\sigma\in \Gamma^\infty(X; \mathcal S\otimes \mathcal V)$.
\end{claim}

Indeed, let $\lambda$ be the constant from the claim. Denote $\Omega_1 = \Omega^\lambda$, then there are constants $k_1$ and $k_2$ such that 
\[  \|\sigma|_{\Omega_1}\| \leq k_1 \|\sigma|_{\Omega}\| + k_2 \|D \sigma\|  \]
for all $\sigma\in \Gamma^\infty(X; \mathcal S\otimes \mathcal V)$. Similarly, let $\Omega_2 = \Omega_1^{\lambda}$, then there are constants $k_3$ and $ k_4$ such that 
 \[  \|\sigma|_{\Omega_2}\| \leq k_3 \|\sigma|_{\Omega_1}\| + k_4 \|D\sigma\|  \]
for all $\sigma\in \Gamma^\infty(X; \mathcal S\otimes \mathcal V)$.

\noindent It follows immediately that 
\begin{align*}
&\|\sigma|_{\Omega_2}\| \leq k_3 \|\sigma|_{\Omega_1}\| + k_4 \|D\sigma\| \\
& \leq k_4 \|D\sigma\| +  k_3 \left(k_1 \|\sigma|_{\Omega}\| + k_2 \|D\sigma\| \right) \\
&= k_1k_3 \|\sigma|_{\Omega}\| +  (k_4 + k_2k_3) \|D\sigma\|
\end{align*}
Inductively, we define $\Omega_{k+1} = \Omega_k^\lambda$. Since $X$ is compact, there exists an integer $n $ such that $\Omega_n = X$. The theorem follows by a finite induction. 
\end{proof}

\begin{remark}
 The constant $\lambda$ is independent of the choice of $\Omega$, although the constants $K_{\Omega, 1}$,  $K_{\Omega, 2}$ and $K_{\Omega, 3}$ may depend on $\Omega$,.  
\end{remark}

Now we shall generalize the above theorem to the case of complete manifolds with positive scalar curvature towards infinity. 

\begin{corollary}\label{cor:sn}
Let $X$ be a spin manifold with a complete Riemannian metric of positive scalar curvature towards infinity. Suppose $\Omega$ is an open subset of $X$ with compact closure. Then there are constants $C_1$ and $C_2$ such that 
\[  \|\sigma\| \leq  C_1 \|\sigma|_\Omega \| + C_2\|D\sigma\|  \] 
for all $\sigma\in \Gamma^\infty(X; \mathcal S\otimes \mathcal V)$.
\end{corollary}
\begin{proof}
Since $X$ has positive scalar curvature towards infinity, there exists an precompact open subset $\Sigma\subset X$ such that 
\begin{enumerate}[(1)]
\item $\Omega \subset\Sigma$, 
\item the scalar curvature $\kappa \geq c_0>0$ on $X - \Sigma$.
\end{enumerate}

The same argument as in the proof of Lemma $\ref{lemma:lb} $ shows that  there exists a constant $c$ such that 
\[  \| D \sigma\| \geq \|(D\sigma)|_{X-\Sigma} \| \geq c\|\sigma|_{X-\Sigma}\| \]
for all $\sigma\in \Gamma^\infty(X; \mathcal S\otimes \mathcal V)$.
Now since the closure of $\Sigma$ is compact, it follows from Theorem $\ref{thm:sn}$ that 
\[  \|\sigma|_\Sigma \| \leq  c_1 \|\sigma|_\Omega \| + c_2\|D\sigma\| \] 
for all $\sigma\in \Gamma^\infty(X; \mathcal S\otimes \mathcal V)$. Therefore we have 
\[  \|\sigma\| \leq \|\sigma|_{X-\Sigma}\| + \|\sigma|_\Sigma \| \leq  C_1 \|\sigma|_\Omega \| + C_2\|D\sigma\|. \] 
\end{proof}

\subsection{Proof of Claim \ref{claim} }

In this subsection, we prove the claim in the proof of Theorem $\ref{thm:sn}$. Our argument is inspired by  the proof of $\cite[Theorem ~ 8.2]{BBKW93}$.

With the same notation from the previous subsection, let $x_0 \in \partial \Omega$. Choose  $r_0 >0$ sufficiently small and $p\in \Omega$ at a distance $r_0$ from $x_0$ such that the ball $B(p_0; r_0 )$ with center at $p_0$ and radius $r_0$ is contained in $\Omega$.  Choose spherical coordinates in a small neighborhood of $p_0$. Denote the ball with center at $p_0$ and radius  $r$ by $B(p_0; r)$. See Figure $\ref{fig:spherical}$ below.
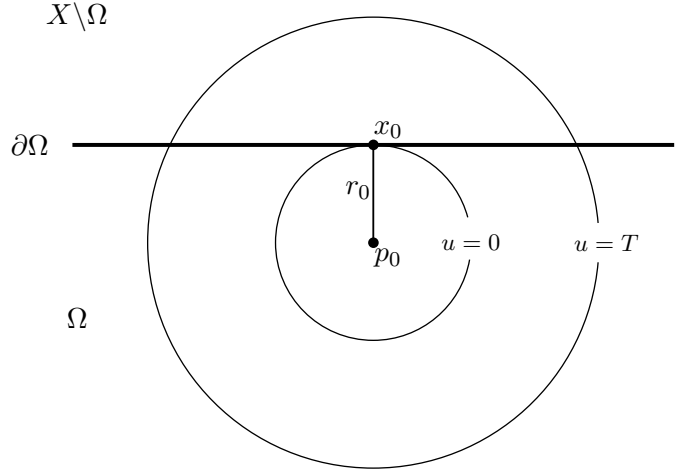
\begin{figure}[h]
\centering
\begin{pspicture*}(-1,-3.5)(8,4)
\psarc[linewidth=.5pt](4, 0){1.3}{15}{350}
\psarc[linewidth=.5pt](4, 0){3}{5}{355}
\psline[linewidth=1.5pt](0,1.3)(8,1.3)
\rput(-0.5,1.3){$\partial \Omega$ }
\rput(0,-1){ $\Omega$}
\rput(0,3){ $X\backslash \Omega$}
\rput(5.3,0){\footnotesize $u=0$}
\rput(7.1,0){\footnotesize $u=T$}
\rput(4.2,1.5){$x_0$}
\psdots[dotsize=4pt](4,1.3)
\psdots[dotsize=4pt](4,0)
\psline[linewidth=.7pt](4,1.3)(4,0)
\rput(3.8,0.7){$r_0$}
\rput(4.2, -0.2){$p_0$}
\end{pspicture*} 
\caption{local spherical coordinates}\label{fig:spherical}
\end{figure}

Let $B = B(p_0; r_0)$   and $\partial B$ its boundary. 
We define an $\mathcal A$-valued inner product 
\[  \lcn \sigma, \eta\rcn_{s} = \int_B \langle (1+ \Delta)^s\sigma(x), \eta(x)\rangle dx \] for all $\sigma, \eta \in \vh^s(B, \mathcal S\otimes \mathcal V)$ (resp. for all $\sigma, \eta \in \vh^s(\partial B, \mathcal S\otimes \mathcal V)$) , where $\Delta$ is the Laplacian operator on $B$ (resp. $\partial B$). 

\begin{lemma}\label{lemma:res}
For $k \geq 1 $,  we have 
\[ \lcn \sigma|_{\partial B}, \sigma|_{\partial B}\rcn_{k-1/2} \leq C \lcn \sigma, \sigma \rcn_k\]
for all $\sigma \in \vh^k(B, \mathcal S \otimes \mathcal V)$.
\end{lemma}
\begin{proof}
See  Appendix $\ref{appendix}$. 
\end{proof}

The following lemma is a generalization of $\cite[Lemma ~8.6]{BBKW93}$. In particular, our argument follows closely the proof of $\cite[Lemma ~8.6]{BBKW93}$.

\begin{lemma}\label{lemma:ineq}
For $R > 0 $ sufficiently large and $T > 0$ sufficiently small, we have 
\begin{align}
& R\int_{u=0}^{T} \int_{\mathbb S_u} e^{R(T-u)^2} \langle \sigma(u, y), \sigma(u, y) \rangle dydu \notag \\
&  \leq C \left(\int_{u=0}^T \int_{\mathbb S_u} e^{R (T-u)^2 } \langle D\sigma, D\sigma\rangle  dy du + RT  e^{RT^2}\int_{\mathbb S_0} \left\langle \sigma, \, \sigma\right\rangle dy  \right)\label{eq:double}
\end{align}
for all $\sigma\in \Gamma^\infty(X, \mathcal S\otimes \mathcal V)$, where $\mathbb S_u$ is the sphere centered at $p_0$ with radius $r_0 + u$.
\end{lemma}

\begin{proof}
In order to simplify the computation, let us consider a few technical points. In the annulus $[0, T]\times \mathbb S_0$, the Dirac operator $D$ takes the form
\[  D = c(u) (\frac{\partial}{\partial u} + \bd_u) \]
where $c(u)$ is the Clifford multiplication of the radial vector. It is easy to see that we may consider the operator $\frac{\partial}{\partial u} + \bd_u$ instead of $D$. We may further deform the Riemannian metric of the manifold and the Hermitian metrics of the bundles on $[0, T] \times \mathbb S_0$ such that they do not depend on the radial direction $u$, but \textit{keep the operator $D$ fixed}. Then it  suffices to prove the lemma under this new metric. The only inconvenience is that in general $\bd_u$ is not a self-adjoint operator with respect to the new structures. This is taken care of by considering its self-adjoint part
\[ \bd_+ = (\bd_u)_+ = \frac{1}{2} (\bd_u + \bd_u^\ast). \]
Notice that $(\bd_u)_+$ is an elliptic differential operator for each $u$, as long as $T$ is sufficiently small. A priori, the choice of $T$ may depend on the neighborhood $B(p_0, r_0)$. However, since $X$ is compact, we can choose $T> 0$ to be independent of $B(p_0, r_0)$. 

Consider $ \nu = e^{R(T-u)^2/2} \sigma$. Then the inequality $\eqref{eq:double}$ becomes 

\begin{align}
&R\int_{u=0}^{T} \int_{\mathbb S_u}  \langle \nu(u, y), \nu(u, y) \rangle dy du\notag  \\
& \leq C \left(\iint \left\langle (D + R(T-u) )\nu,  (D + R(T-u) )\nu\right\rangle  dy du + RT \int_{\mathbb S_0} \left\langle \nu, \, \nu\right\rangle dy \right)
\end{align}
Decompose $\frac{\partial}{\partial u} + \bd + R(T-u)$ into its symmetric part $\bd_+ + R(T-u)$ and its anti-symmetric part $\frac{\partial}{\partial u} + \bd_-$ with 
\[  \bd_- = (\bd_u)_- = \frac{1}{2}(\bd_u - \bd_u^\ast). \]
Then 
\begin{align}
& \int_{u=0}^T \int_{\mathbb S_u} \left\langle (D + R(T-u) )\nu,  (D + R(T-u) )\nu\right\rangle  dy du \notag\\
& =\int_{u=0}^T \int_{\mathbb S_u} \left\langle \frac{\partial \nu}{\partial u}+ \bd \nu + R(T-u)\nu, \frac{\partial \nu}{\partial u}+ \bd \nu + R(T-u) \nu\right\rangle  dy du \notag \\
& = \iint \left\langle\frac{\partial \nu}{\partial u}+ \bd_- \nu, \frac{\partial \nu}{\partial u}+ \bd_- \nu \right\rangle  dy du \notag\\
& \quad + \iint \langle (\bd_+ + R(T-u)) \nu, (\bd_+ + R(T-u)) \nu\rangle dydu \notag\\
& \quad +  \iint \left\langle \frac{\partial \nu}{\partial u} + \bd_-\nu, \, \bd_+\nu + R(T-u) \nu\right\rangle dy du \notag\\
&  \quad + \iint \left\langle \bd_+\nu + R(T-u) \nu, \, \frac{\partial \nu}{\partial u} + \bd_-\nu\right\rangle dy du \label{eq:bbd}
\end{align}

Let us consider the last two terms of $\eqref{eq:bbd}$. By integration by parts, we have 
\begin{align}
& \iint \left\langle \frac{\partial \nu}{\partial u} + \bd_-\nu, \, \bd_+\nu + R(T-u) \nu\right\rangle dy du \notag\\
& + \iint \left\langle \bd_+\nu + R(T-u) \nu, \, \frac{\partial \nu}{\partial u} + \bd_-\nu\right\rangle dy du \notag \displaybreak[0]\\
& = \iint \left\langle \frac{\partial \nu}{\partial u}, \, \bd_+\nu  + R(T-u)  \nu\right\rangle dy du  \notag\\ 
& \quad + \iint \left\langle  \bd_+\nu + R(T-u)  \nu, \, \frac{\partial \nu}{\partial u} \right\rangle dy du \notag\\
&  \quad +  \iint \left\langle \bd_-\nu, \, \bd_+\nu \right\rangle dy du  + \iint \left\langle \bd_+\nu, \, \bd_-\nu \right\rangle dy du  \displaybreak[0]\notag\\
& =  \int_{\mathbb S_T} \left\langle \nu, \, \bd_+\nu\right\rangle dy- \int_{\mathbb S_0} \left\langle \nu, \, \bd_+\nu + RT \nu\right\rangle dy \notag \\
& \quad -\iint \left\langle \nu, \, \left(\frac{\partial}{\partial u}(\bd_+\nu + R(T-u))\right)  \nu\right\rangle dy du \notag \\
& \quad + \iint \left\langle  \nu, \,  (\bd_+ + R(T-u)) \frac{\partial \nu}{\partial u} \right\rangle dy du + \iint \left\langle \nu, \, [\bd_+, \bd_-]\nu \right\rangle dy du \displaybreak[1]\notag\\
& =  \int_{\mathbb S_T} \left\langle \nu, \, \bd_+\nu\right\rangle dy -  \int_{\mathbb S_0} \left\langle \nu, \, \bd_+\nu + RT  \nu\right\rangle dy  \notag\\
& \quad +  \iint \left\langle \nu, \, -\frac{\partial \bd_+}{\partial u}\nu + R\nu \right\rangle dy du  + \iint \left\langle \nu, \, [\bd_+, \bd_-]\nu \right\rangle dy du \displaybreak[0]\notag\\
& =  \int_{\mathbb S_T} \left\langle \nu, \, \bd_+\nu\right\rangle dy -  \int_{\mathbb S_0} \left\langle \nu, \, \bd_+\nu \right\rangle dy  -  RT\int_{\mathbb S_0} \left\langle \nu, \,\nu \right\rangle dy\notag\\
& \quad + R\iint \langle \nu, \nu \rangle dy\ du   + \iint \left\langle \nu, \, -\frac{\partial \bd_+}{\partial u}\nu + [\bd_+, \bd_-] \nu \right\rangle dy du. \label{eq:1}
\end{align}

Now we prove the lemma in three steps. 

\noindent\textbf{Step One:} We shall prove that  
\begin{align}
& \pm \int_0^T\int_{\mathbb S_u} \left\langle \nu, \, -\frac{\partial \bd_+}{\partial u}\nu + [\bd_+, \bd_-] \nu \right\rangle dy du \notag\\
&\leq k \left( R \iint \langle \nu, \nu \rangle dy du \right. \notag\\
      &\hspace{2cm}    + \left.\iint \langle (\bd_+ + R(T-u))\nu, (\bd_+ + R(T-u))\nu\rangle  dy du \right) \label{eq:2}
\end{align}
for some constant $ 0 < k < 1$.

Notice that $a^\ast b + b^\ast a \leq  \lambda a^\ast a + \frac{1}{\lambda } b^\ast b$
for all $a, b\in \mathcal A$ and for all $\lambda >0$.  
Moreover, since the operators $(\bd_u)_+$ are first order elliptic operators,  we have the following $\mathcal A$-valued G\aa rding's inequality (cf. Lemma $\ref{lemma:garding}$):
\[   \lcn f, f\rcn_1 \leq c \lcn f, f \rcn_0  + c \lcn (\bd_u)_+ f,  (\bd_u)_+ f\rcn_0
\] 
for all $f\in \Gamma^\infty(\mathbb S_u, \mathcal S\otimes \mathcal V)$. Therefore, we have
\begin{align*}
& 2 \iint \left\langle \nu, \, -\frac{\partial \bd_+}{\partial u}\nu + [\bd_+, \bd_-] \nu \right\rangle dy du \\
&\leq  \lambda \iint \langle \nu, \nu \rangle dy du \\
& \quad + \frac{1}{\lambda}\iint \left\langle  (-\frac{\partial \bd_+}{\partial u} + [\bd_+, \bd_-]) \nu,  (-\frac{\partial \bd_+}{\partial u} + [\bd_+, \bd_-]) \nu \right\rangle  dy du \displaybreak[1]\\
&\leq  \lambda \iint \langle \nu, \nu \rangle dy du + \frac{c_1}{\lambda}\int \lcn \nu, \nu \rcn_1 du\\
&\leq  \lambda \iint \langle \nu, \nu \rangle dy du + \frac{c_1 c}{\lambda}  \int \lcn\nu, \nu \rcn_0 + \lcn \bd_+ \nu, \bd_+ \nu \rcn_0 \, du \\
&= \lambda \iint \langle \nu, \nu \rangle dy du + \frac{c_1 c}{\lambda} \left( \iint \left\langle \nu, \nu \right\rangle  dy du + \iint \left\langle \bd_+\nu, \bd_+\nu \right\rangle  dy du\right) \displaybreak[1]\\
& \leq  \left( \lambda  + \frac{c_1 c}{\lambda}\right) \iint \langle \nu, \nu \rangle dy du + \frac{2 c_1 c}{\lambda} \iint \left\langle  R(T-u)\nu,   R(T-u)\nu \right\rangle  dy du \\
&\quad + \frac{2 c_1 c}{\lambda} \iint \left\langle (\bd_+ + R(T-u))\nu, (\bd_+ + R(T-u))\nu \right\rangle  dy du \\
& \leq \left( \lambda  + \frac{c_1 c}{\lambda}(2R^2T^2 +1) \right) \iint \langle \nu, \nu \rangle dy du \\
& \quad + \frac{2 c_1 c}{\lambda} \iint \left\langle (\bd_+ + R(T-u))\nu, (\bd_+ + R(T-u))\nu \right\rangle  dy du.
\end{align*}
Choose $\lambda = R$, then 
\begin{align*}
& \iint \left\langle \nu, \, -\frac{\partial \bd_+}{\partial u}\nu + [\bd_+, \bd_-] \nu \right\rangle dy du \\
& \leq 
R \left( \frac{1}{2}  + c_1 c T^2 + \frac{c_1 c}{2R^2} \right) \iint \langle \nu, \nu \rangle dy du \\
& \quad + \frac{c_1 c}{R} \iint \left\langle (\bd_+ + R(T-u))\nu, (\bd_+ + R(T-u))\nu \right\rangle  dy du.
\end{align*}
This proves $\eqref{eq:2}$ for $R$ sufficiently large and $T$ sufficiently small. Note that the constants $c_1$ and $c$ depend on the local spherical coordinate chart. Since the manifold $X$ is compact, we see that $c_1$ and $c$ are uniformly bounded on $X$. Therefore the choice of the constant $T$ can be made independent of the local neighborhood $B(p_0, r_0)$.

\noindent \textbf{Step Two.} Now let us consider the term 
\[ \int_{\mathbb S_T} \left\langle \nu, \, \bd_+\nu\right\rangle dy - \int_{\mathbb S_0} \left\langle \nu, \, \bd_+\nu\right\rangle dy. \]
Recall that  $\nu = e^{R(T-u)^2/2} \sigma$. In particular, $\nu = \sigma$ on $\mathbb S_T$. It follows that 
\begin{align*}
&\int_{\mathbb S_T} \left\langle \nu, \, \bd_+\nu\right\rangle dy  - \int_{\mathbb S_0} \left\langle \nu, \, \bd_+\nu\right\rangle dy\\
& = \int_{\mathbb S_T} \left\langle \sigma, \, \bd_+\sigma\right\rangle dy - \int_{\mathbb S_0} \left\langle \sigma, \, \bd_+\sigma\right\rangle dy\\
& = \int_0^T \int_{\mathbb S_u} \frac{\partial }{\partial u}\left\langle \sigma, \, \bd_+\sigma\right\rangle dy du \displaybreak[1]\\
&  = \int_0^T \int_{\mathbb S_u} \left\langle \frac{\partial \sigma}{\partial u}, \, \bd_+\sigma\right\rangle dy du + \int_0^T \int_{\mathbb S_u} \left\langle \sigma, \, \frac{\partial \bd_+}{\partial u}\sigma\right\rangle dy du \\
& \quad +\int_0^T \int_{\mathbb S_u} \left\langle \sigma, \, \bd_+ \frac{\partial \sigma}{\partial u}\right\rangle dy du.  
\end{align*}
Since $\bd_+$ is self-adjoint, we have 
\[ \int_0^T \int_{\mathbb S_u} \left\langle \sigma, \, \bd_+ \frac{\partial \sigma}{\partial u}\right\rangle dy du = \int_0^T \int_{\mathbb S_u} \left\langle \bd_+ \sigma, \,  \frac{\partial \sigma}{\partial u}\right\rangle dy du. \]
By Lemma $\ref{lemma:bbd}$ and Lemma $\ref{lemma:garding}$, we see that 
\begin{align*}
& \int_0^T \int_{\mathbb S_u} \left\langle \frac{\partial \sigma}{\partial u}, \, \bd_+\sigma\right\rangle dy du + \int_0^T \int_{\mathbb S_u} \left\langle \bd_+ \sigma, \,  \frac{\partial \sigma}{\partial u}\right\rangle dy du \\
&\leq   \int_0^T \int_{\mathbb S_u}  \left\langle \frac{\partial \sigma}{\partial u}, \, \frac{\partial \sigma}{\partial u}\right\rangle dy du + \int_0^T \int_{\mathbb S_u} \left\langle \bd_+ \sigma, \,  \bd_+\sigma\right\rangle dy du \\
& \leq K_0\int_0^T \int_{\mathbb S_u}  \langle \sigma, \sigma \rangle dy du + K_0\int_0^T \int_{\mathbb S_u}  \langle D\sigma, D\sigma \rangle dydu.
\end{align*}
Similarly, we have 
\begin{align*}
& \int_0^T \int_{\mathbb S_u} \left\langle \sigma, \, \frac{\partial \bd_+}{\partial u}\sigma\right\rangle dy du \\
&\leq K_1\int_0^T \int_{\mathbb S_u}  \langle \sigma, \sigma \rangle dy du + K_1\int_0^T \int_{\mathbb S_u}  \langle D\sigma, D\sigma \rangle dydu. 
\end{align*}
It follows that 
\begin{align*}
&\int_{\mathbb S_T} \left\langle \nu, \, \bd_+\nu\right\rangle dy - \int_{\mathbb S_0} \left\langle \nu, \, \bd_+\nu\right\rangle dy\\
&\leq K\int_0^T \int_{\mathbb S_u}  \langle \sigma, \sigma \rangle dy du + K\int_0^T \int_{\mathbb S_u}  \langle D\sigma, D\sigma \rangle dydu. 
\end{align*}
In fact, the same argument shows that 
\begin{align}
&\pm \left(\int_{\mathbb S_T} \left\langle \nu, \, \bd_+\nu\right\rangle dy - \int_{\mathbb S_0} \left\langle \nu, \, \bd_+\nu\right\rangle dy\right) \notag \\
&\leq K\int_0^T \int_{\mathbb S_u}  \langle \sigma, \sigma \rangle dy du + K\int_0^T \int_{\mathbb S_u}  \langle D\sigma, D\sigma \rangle dydu.  \label{eq:3}
\end{align}

\noindent \textbf{Step Three.} Combining $\eqref{eq:bbd}$, $\eqref{eq:1}$ and $\eqref{eq:2}$ together, we have  
\begin{align}
& \int_{u=0}^T \int_{\mathbb S_u} \left\langle (D + R(T-u) )\nu,  (D + R(T-u) )\nu\right\rangle  dy du \notag \displaybreak[1]\\
& \geq  \iint \left\langle\frac{\partial \nu}{\partial u}+ \bd_- \nu, \frac{\partial \nu}{\partial u}+ \bd_- \nu \right\rangle  dy du \notag\\
& \quad + (1-k) \iint \langle (\bd_+ + R(T-u)) \nu, (\bd_+ + R(T-u)) \nu\rangle dydu \notag\\
& \quad  + (1-k) R\iint \langle \nu, \nu \rangle dy\ du - RT\int_{\mathbb S_0} \left\langle \nu, \,  \nu\right\rangle dy  \notag\\
& \quad + \int_{\mathbb S_T} \left\langle \nu, \, \bd_+\nu\right\rangle dy -  \int_{\mathbb S_0} \left\langle \nu, \, \bd_+\nu \right\rangle dy. 
\end{align}
It follows that 
\begin{align*}
& (1-k) R\int_0^T\int_{\mathbb S_u} \langle \nu, \nu \rangle dydu \\
&\leq \int_{u=0}^T \int_{\mathbb S_u} \left\langle (D + R(T-u) )\nu,  (D + R(T-u) )\nu\right\rangle  dy du \\
& \quad  +  RT \int_{\mathbb S_0} \left\langle \nu, \,   \nu\right\rangle dy + \int_{\mathbb S_0} \left\langle \nu, \, \bd_+\nu\right\rangle dy - \int_{\mathbb S_T} \left\langle \nu, \, \bd_+\nu\right\rangle dy.
\end{align*}
Recall that $\nu = e^{R(T-u)^2/2} \sigma$. By applying $\eqref{eq:3}$ to the above inequality, we have  
\begin{align*}
& (1-k) R\int_0^T\int_{\mathbb S_u} e^{R(T-u)^2}\langle \sigma, \sigma \rangle dydu \\
&\leq  \int_{u=0}^T \int_{\mathbb S_u} e^{R(T-u)^2}\left\langle D\sigma, D\sigma\right\rangle  dy du +  RT e^{RT^2}\int_{\mathbb S_0} \left\langle \sigma, \,   \sigma\right\rangle dy \\
& \quad  + K\int_0^T \int_{\mathbb S_u}  \langle \sigma, \sigma \rangle dy du + K\int_0^T \int_{\mathbb S_u}  \langle D\sigma, D\sigma \rangle dydu. 
\end{align*}
It follows immediately that 
\begin{align*}
& [(1-k) R - K]\int_0^T\int_{\mathbb S_u} e^{R(T-u)^2}\langle \sigma, \sigma \rangle dydu \\
&\leq (1+ K) \int_{u=0}^T \int_{\mathbb S_u} e^{R(T-u)^2}\left\langle D\sigma, D\sigma\right\rangle  dy du +  RT e^{RT^2}\int_{\mathbb S_0} \left\langle \sigma, \,   \sigma\right\rangle dy. 
\end{align*}
The proof is finished by choosing $R$ sufficiently large.
\end{proof}

Now we use the above lemmas to prove Claim $\ref{claim}$.  

\begin{proof}[\textbf{Proof of Claim \ref{claim}}]
Recall that $T$ in Lemma $\ref{lemma:ineq}$ can be chosen independent of the local small neighborhoods. Now since the boundary of $\Omega$ is compact, then for all sufficiently small $\lambda$, we have   
\[  \Omega^\lambda \subset \Omega \cup \bigcup_{i = 0}^{N} B\left(p_i; r_i + T\right)  \]
for some $N\in \mathbb N$.  Therefore, it suffices to show that we have constants $C_1$ and $ C_2$ such that 
\[  \|\sigma|_{\Omega'}\| \leq  C_1 \|\sigma|_{\Omega}\| + C_2 \|D\sigma\|   \]
for all $\sigma\in \Gamma^\infty(X; \mathcal S\otimes \mathcal V)$, where $\Omega' =  \Omega\cup B(p_0; r_0+ T)$. 

By Lemma $\ref{lemma:ineq}$, we have 
\begin{align*}
& \int_{\Omega'} \langle \sigma, \sigma \rangle dg \leq \int_\Omega \langle \sigma, \sigma \rangle dg + \int_{B(p_0;\, r_0+ T)} \langle \sigma, \sigma \rangle dg  \\ 
&  = \int_\Omega \langle \sigma, \sigma \rangle dg + \int_{B(p_0; r_0)} \langle \sigma, \sigma \rangle dg + \int_{u=0}^{T} \int_{\mathbb S_u} \langle \sigma, \sigma \rangle dy du \\
& \leq 2\int_\Omega \langle \sigma, \sigma \rangle dg + \int_{u=0}^T \int_{\mathbb S_u} e^{R(T-u)^2}\langle \sigma, \sigma \rangle dy du \displaybreak[2]\\
& \leq 2\int_\Omega \langle \sigma, \sigma \rangle dg  \\
& \quad + \frac{C}{R} \left(\int_{u=0}^T \int_{\mathbb S_u} e^{R (T-u)^2 } \langle D\sigma, D\sigma\rangle  dy du + RT e^{RT^2}\int_{\mathbb S_0} \left\langle \sigma, \, \sigma\right\rangle dy \right)\\
& \leq 2\int_\Omega \langle \sigma, \sigma \rangle dg + \frac{Ce^{RT^2}}{R} \left(\int_{u=0}^T \int_{\mathbb S_u} \langle D\sigma, D\sigma\rangle  dy du + RT\int_{\mathbb S_0} \left\langle \sigma, \, \sigma\right\rangle dy \right)
\end{align*}
for $R$ sufficiently large. Here $dg$ stands for the volume form on $X$. 
Now by Lemma $\ref{lemma:res}$, it follows that  
\[ \|\sigma|_{\Omega'}\|^2 \leq K_1 \|\sigma|_\Omega\|^2 + K_2 \|D\sigma\|^2 \]
for some constants $K_1$ and $K_2$. This finishes the proof.

\end{proof}

\section{Diffeomorphisms and positive scalar curvature}\label{sec:diffeo}

In this section, we apply our relative higher index theorem to study $\mathcal R^+(M)$ the space of all metrics of positive scalar curvature on a manifold $M$. All the results and their proofs in this section work for both the real and the complex cases. For simplicity, we only state and prove the results for the complex case.   

Throughout this section, we assume that $M$ is an odd dimensional\footnote{In the real case, assume that $\dim M = -1 \pmod 8$.} closed spin manifold and $M_\Gamma$ an $\Gamma$-cover of $M$, where $\Gamma$ is a discrete group. Assume $M$ carries positive scalar curvature, i.e. $\mathcal R^+(M)\neq \emptyset$. Choose $g_0, g_1\in \mathcal R^+(M)$. We define a smooth path of Riemannian metrics $g_t$  on $M$ such that 
\[ 
g_t = \begin{cases}
g_0 \quad \textup{for $t\leq 0$, }\\
g_1 \quad \textup{for $t\geq 1$,} \\
\textup{any smooth homotopy from $g_0$ to $g_1$ for $0\leq t \leq 1$.}
\end{cases}\] 
Then $X = M\times \mathbb R$ endowed with the metric $ h = g_t + (dt)^2$ becomes a complete Riemannian manifold with positive scalar curvature towards infinity. 

Denote $X_\Gamma = M_\Gamma \times \mathbb R$. Then $X_\Gamma$ is naturally a $\Gamma$-cover of $X$ with $\Gamma$ acting on $\mathbb R$ trivially. We define a flat $C_r^\ast(\Gamma)$-bundle $\mathcal V$ on $X$ by
\[  \mathcal V = X_\Gamma\times_\Gamma C_r^\ast(\Gamma). \]
Let $\mathcal S = \mathcal S^+ \oplus \mathcal S^-$ the spinor bundle over $X$. Then, with the flat connection on $\mathcal V$, we can define the Dirac operator 
\begin{equation} \label{eq:md}
 D_{\mathcal V} : \Gamma^\infty (X, \mathcal S\otimes \mathcal V) \to \Gamma^\infty (X, \mathcal S\otimes \mathcal V). 
\end{equation}
By the discussion in Section $\ref{sec:dirac}$, we have a higher index class $\ind(D_{\mathcal V} ) \in K_0(C_r^\ast(\Gamma)).$ We also write $\ind(D_{\mathcal V} ) = \ind_\Gamma(g_0, g_1)$ if we want to specify the metrics. 

\begin{question}
It remains an open question whether $\ind_\Gamma(g_0, g_1)$ lies in the image of the Baum-Connes assembly map $ \mu: K_0^{\Gamma}(\underline{E}\Gamma) \to K_0(C^\ast_r(\Gamma)).$
\end{question}
We refer the reader to  \cite{PBAC88} and \cite{BCH94} for a detailed description of the Baum-Connes assembly map (\cite{PBMK04} for its real analogue).

Now let $g_0, g_1, g_2 \in \mathcal R^+(M)$ be Riemannian metrics of positive scalar curvature on $M$. The following propositions generalize the corresponding classical results of Gromov and Lawson \cite[Theorem 4.41 \& Theorem 4.48]{MGBL83}. 
\begin{proposition}\label{prop:rd}
\[ \ind_\Gamma(g_0, g_1) + \ind_\Gamma(g_1, g_2) = \ind_\Gamma(g_0, g_2).\]
\end{proposition}
\begin{proof}
 The statement follows immediately from the relative higher index theorem (Theorem $\ref{thm:rit}$).  
\end{proof}

 Denote by $\textup{Diff}^\infty(M)$ the group of diffeomorphisms on $M$. For a fixed metric $g\in \mathcal R^+(M)$, set 
\[ \ind_\Gamma(\diffeo) = \ind_\Gamma(g, (\diffeo^{-1})^\ast g) \]
for $\diffeo\in \textup{Diff}^\infty(M)$. 

Recall that in the case when $\Gamma = \pi_1(M)$ the fundamental group of $M$, there is a natural homomorphism 
\[ \varphi: \textup{MCG}(M) = \textup{Diff}^\infty(M)/\textup{Diff}_0^\infty(M) \to \textup{Out}(\Gamma) \]  
where $\textup{Out}(\Gamma)$ is the group of outer automorphisms of $\Gamma$ and $\textup{Diff}^\infty_0(M)$ is  the connected component of the identity in $\textup{Diff}^\infty(M)$. In particular, each $\diffeo \in \textup{Diff}^\infty(M)$ induces an automorphism $\diffeo_\ast:  K_0(C_r^\ast(\Gamma)) \to  K_0(C_r^\ast(\Gamma))$. We denote by $K_0(C_r^\ast(\Gamma))\rtimes_\varphi \textup{MCG}(M)$ the semi-direct product of $K_0(C_r^\ast(\Gamma))$ and $\textup{MCG}(M)$, where $\textup{MCG}(M)$ acts on $K_0(C_r^\ast(\Gamma))$ through $\varphi$.
\begin{proposition}
For $\Gamma = \pi_1(M)$, we have a group homomorphism 
\[ \ind_\Gamma: \textup{Diff}^\infty(M)/\textup{Diff}^\infty_0(M) \to K_0(C_r^\ast(\Gamma))\rtimes_\varphi \textup{MCG}(M). \]

\end{proposition}
\begin{proof}
Note that 
\[ \diffeo_\ast[\ind_\Gamma(g_0, g_1)] =  \ind_\Gamma((\diffeo^{-1})^\ast g_0, (\diffeo^{-1})^\ast g_1) \]
for all $\diffeo\in \textup{Diff}^\infty(M)$. It follows that 
\begin{align*}
\ind_\Gamma(\diffeo_1\circ \diffeo_2) & = \ind_\Gamma(g,\, (\diffeo_1^{-1})^\ast \circ (\diffeo_2^{-1})^\ast g) \\
& =  \ind_\Gamma(g,\, (\diffeo_1^{-1})^\ast g) +  \ind_\Gamma((\diffeo_1^{-1})^\ast g,\, (\diffeo_1^{-1})^\ast \circ (\diffeo_2^{-1})^\ast g) \\
& = \ind_\Gamma(g,\, (\diffeo_1^{-1})^\ast g) + \diffeo_\ast[ \ind_\Gamma( g,\,  (\diffeo_2^{-1})^\ast g)].
\end{align*}
Clearly, the map $\ind_\Gamma$ is trivial on $\textup{Diff}^\infty_0(M)$. Hence follows the lemma. 
\end{proof}

\subsection{Applications}

For the rest of the section, we fix a Riemannian metric $g_0\in \mathcal R^+(M)$ and fix a spin structure on $M$. Let $\diffeo\in \textup{Diff}^\infty(M)$. Assume that $\diffeo$  preserves the orientation of $M$ and the spin structure on $M$.  Let  $g_1 = (\diffeo^{-1})^\ast g_0 $.

From now on, let $\Gamma = \pi_1(M)$. Consider the mapping cylinder $M_\diffeo = (M\times [0, 1])/\sim,$  where $\sim$ is the equivalence relation $(x, 0)\sim (\diffeo(x), 1)$ for $x\in M$. Now $M_{\Gamma} = \widetilde M$ is the universal cover of $M$. Note that $\diffeo$ induces an outer automorphism $\diffeo_\ast \in \textup{Out}(\Gamma)$. More precisely, $\diffeo_\ast: \Gamma \to \Gamma$ is only well defined modulo inner automorphisms. In the following, we fix a representative in the class of this outer automorphism. We shall see that our results below, which are stated at the level of $K$-theory, do not depend on such a choice.   

Let $\Gamma\rtimes_\diffeo \mathbb Z$ be the semi-direct product with the action $\mathbb Z$ on $\Gamma$ induced by $\diffeo$. We shall simply write $\Gamma\rtimes\mathbb Z$ for $\Gamma\rtimes_\varPsi\mathbb Z$ if no ambiguity arises. We see that  
$ M_\Gamma \times \mathbb R$ is a $(\Gamma\rtimes \mathbb Z)$-cover of $M_\diffeo$.

Consider the natural inclusion $\iota:  \Gamma \hookrightarrow \Gamma\rtimes \mathbb Z$. We denote the induced inclusion map on $C^\ast$-algebras also by $\iota: C_r^\ast(\Gamma) \to C_r^\ast(\Gamma\rtimes \mathbb Z).$
 Then we have a homomorphism 
\[ \iota_\ast : K_0(C_r^\ast(\Gamma)) \to K_0(C_r^\ast(\Gamma\rtimes \mathbb Z)).\]
Note that an inner automorphism of $\Gamma$ induces an inner automorphism of $C_r^\ast(\Gamma\rtimes \mathbb Z)$, hence its induced automorphism  on  $K_0(C_r^\ast(\Gamma\rtimes \mathbb Z))$ is the identity map. It follows that  $\iota_\ast$ is independent of the choice of the representative for the outer automorphism class $\diffeo_\ast \in \textup{Out} (\Gamma)$. 

Before we prove the main result of this section,  let us fix some notation. Let $g_t$  be a smooth path of Riemannian metrics on $M$ such that 
\[ 
g_t = \begin{cases}
g_0 \quad \textup{for $t\leq 0$, }\\
g_1 = (\diffeo^{-1})^\ast g_0 \quad \textup{for $t\geq 1$,} \\
\textup{any smooth homotopy from $g_0$ to $g_1$ for $0\leq t \leq 1$.}
\end{cases}\] 
We consider the following list of Dirac operators and their index classes. 
\begin{enumerate}[(a)]
\item For the manifold $X = M\times \mathbb R$ with the Riemannian metric $h = g_t + (dt)^2$, we denote by $D_{\mathcal V}$ its Dirac operator with coefficients in  $\mathcal V= X_\Gamma\times_{\Gamma} C_r^\ast(\Gamma)$.

\item We denote the same manifold $M\times \mathbb R$ but with the metric $(\diffeo^{-n})^\ast h$ by $X_n$. Let $\mathcal V_n = \mathcal V$ be the corresponding flat $C_r^\ast(\Gamma)$-bundle over $X_n$. Then $\mathbb Z$  acts isometrically on the disjoint union $X_{\mathbb Z}  = \bigcup_{n\in \mathbb Z} X_n$ by 
\[n \mapsto \diffeo^n : X_k \to X_{k+n}.\] The action of $\mathbb Z$ actually lifts to an action on $\mathcal V_{\mathbb Z}  = \bigcup_{n\in \mathbb Z} \mathcal V_n.$ Equivalently, we consider the following flat $C^\ast_r(\Gamma\rtimes \mathbb Z)$-bundle over $X$:  
\[  \mathcal W = \mathcal V_{\mathbb Z}\times_{\mathbb Z} C_r^\ast(\mathbb Z). \]
With the metric $h$ on $X$, we denote by $D_0 = D_{\mathcal W}$ the associated Dirac operator on $X$ with coefficients in $\mathcal W$. Then 
\[  \iota_\ast(\ind(D_{\mathcal V}) ) = \ind(D_0) \]
in $K_0(C_r^\ast(\Gamma\rtimes \mathbb Z))$.

\item We denote the manifold $M\times \mathbb R$ but with the product metric $g_0 + (dt)^2$ by $X'$. Then the same construction from $(b)$ produces a Dirac operator $D_1 = D_{\mathcal W'}$ on $X'$. Since $X'$ can be viewed as a double, it follows from Theorem $\ref{thm:double}$ that $\ind(D_1) = 0 $ in $K_0(C_r^\ast(\Gamma\rtimes \mathbb Z))$. In fact, one can directly show that $\ind(D_1) = 0 $ without referring to Theorem $\ref{thm:double}$.  Indeed, since $g_0$ has positive scalar curvature, $g_0 + (dt)^2$ also has positive scalar curvature everywhere on $X'$. This immediately implies that $D_1$ is bounded below. Therefore $\ind(D_1) = 0 $.

\item  Let $\mathcal S_\diffeo$ be the spinor bundle over the mapping cylinder $M_\diffeo$. Define a flat $C_r^\ast(\Gamma\rtimes \mathbb Z)$-bundle over $M_\diffeo$ by
\[ \mathcal W_\diffeo = (M_\Gamma\times \mathbb R)\times_{\Gamma\rtimes \mathbb Z} C_r^\ast(\Gamma\rtimes \mathbb Z).  \]
We denote the associated Dirac operator by $D_2 = D_{\mathcal W_\diffeo}$. Note that, since $M_\diffeo$ is closed, $\ind(D_2)$ does not depend on which Riemannian metric we have on $M_\diffeo$.  Equivalently, we view $M_\Gamma\times \mathbb R$ as a $\Gamma\rtimes \mathbb Z$-cover of $M_\diffeo$. Then the Dirac operator $ D_{\Gamma\rtimes \mathbb Z}$ on $M_\Gamma\times \mathbb R$ defines a higher index class $\ind(D_{\Gamma\rtimes \mathbb Z})$ in $K_0(C_r^\ast(\Gamma\rtimes \mathbb Z))$ (cf. \cite[Section 5]{CM90}). By construction, we have $ \ind(D_2)  = \ind(D_{\Gamma\rtimes \mathbb Z}).$

\end{enumerate}

Recall that we write $\ind(D_{\mathcal V}) = \ind_\Gamma(\diffeo)$.  The following theorem provides a formula in many cases to determine when $\iota_\ast( \ind_\Gamma(\diffeo) )$ is nonvanishing (e.g. when the strong Novikov conjecture holds for $\Gamma\rtimes \mathbb Z$).

\begin{theorem}\label{thm:hc}
With the above notation, we have 
 \[  \iota_\ast( \ind_\Gamma(\diffeo ) )= \ind(D_{\Gamma\rtimes \mathbb Z})  \]
in $ K_0(C_r^\ast(\Gamma\rtimes \mathbb Z))$. 
In particular, this implies that $\iota_\ast( \ind_\Gamma(\diffeo ) )$ lies in the image of the Baum-Connes assembly map
\[ \mu: K_0^{\Gamma\rtimes \mathbb Z}(\underline{E}(\Gamma\rtimes\mathbb Z)) \to K_0(C^\ast_r(\Gamma\rtimes \mathbb Z)). \]
\end{theorem} 
\begin{remark}
If $m = \dim M \neq -1 \pmod 8$ in the real case (resp. if $m$ is even in the complex case), we consider the $\cl_n$-linear Dirac operator as in Remark $\ref{remark:dim}$. The same proof below implies that  $ \iota_\ast( \ind_\Gamma(\diffeo) ) = \ind(D_{\Gamma\rtimes \mathbb Z})$ in $  K_{m+1}(C_r^\ast(\Gamma\rtimes \mathbb Z; \mathbb R))$ (resp. $ K_{1}(C_r^\ast(\Gamma\rtimes \mathbb Z))$). Again, in the complex case, one can in fact apply the above theorem to $\mathbb S^1\times M$ to cover the case of even dimensional manifolds.

\end{remark}

\begin{proof}[Proof of Theorem $\ref{thm:hc}$]
Notice that the left end $\Omega'_l$ (resp. the right end $\Omega'_r$) of $X'$ is isometric to the left end $\Omega_l$ (resp. the right end $\Omega_r$) of $X$ through the identity map (resp. the diffeomorphism $\diffeo$). Moreover, the isometries lift to isometries from $\mathcal W|_{\Omega'_l} $ to $\mathcal W_{\Omega_l}$ (resp. from $\mathcal W|_{\Omega'_r} $ to $\mathcal W_{\Omega_r}$). So we can apply our relative higher index theorem (Theorem $\ref{thm:rit}$) to $D_0$ and $D_1$, where we identify $\Omega_0 = \Omega'_l\cup \Omega'_r$ with $\Omega_1 = \Omega_l \cup \Omega_r$.   

Following the cutting-pasting procedure in Section $\ref{sec:rel}$, we see that $X'$ and $X$ join together to give exactly $M_\diffeo$. Moreover, the $C_r^\ast(\Gamma\rtimes \mathbb Z)$-bundles $\mathcal W'$ and $\mathcal W$ join together to give precisely the bundle $\mathcal W_\diffeo$ over $M_\diffeo$. Therefore, by Theorem $\ref{thm:rit}$, we have  
\[ \ind(D_2) = \ind(D_0) - \ind(D_1). \]
By the discussion above, we have $\ind(D_0) = \iota_\ast( \ind_\Gamma(\diffeo ) )$, $\ind(D_1) = 0$ and $\ind(D_2) =  \ind(D_{\Gamma\rtimes \mathbb Z})$. This finishes the proof.

\end{proof}

\begin{question}
It remains an open question whether $\ind_\Gamma(\diffeo)$ lies in the image of the Baum-Connes assembly map \[ \mu: K_0^{\Gamma}(\underline{E}\Gamma) \to K_0(C^\ast_r(\Gamma)). \]
\end{question}

In the following, we show that this question has an affirmative answer for some special cases. 

Recall that $\diffeo$ induces an automorphism $\diffeo_\ast: \Gamma\to \Gamma$ (up to inner automorphisms), thus an automorphism $\diffeo_\ast: C_r^\ast(\Gamma) \to C_r^\ast(\Gamma)$ (up to inner automorphisms). Therefore, we have well-defined isomorphisms  
\[\diffeo_\ast: K_i(C_r^\ast(\Gamma)) \to K_i(C_r^\ast(\Gamma)),\quad  i= 0, 1. \]
\begin{corollary}
If $\diffeo_\ast = \textup{Id}: \Gamma \to \Gamma$, then $\ind_\Gamma(\diffeo)$ lies in the image of the Baum-Connes assembly map $\mu: K_0^{\Gamma}(\underline{E}\Gamma) \to K_0(C^\ast_r(\Gamma)).$
\end{corollary}
\begin{proof}
Since $\diffeo_\ast = \textup{Id}: \Gamma \to \Gamma$, it follows immediately that 
\[ C_r^\ast(\Gamma\rtimes \mathbb Z) \cong C_r^\ast(\Gamma) \otimes C_r^\ast(\mathbb Z). \]
Then we have 
\[ K_0(C_r^\ast(\Gamma\rtimes \mathbb Z)) \cong K_0(C_r^\ast(\Gamma)) \oplus K_1(C_r^\ast(\Gamma)).\]
Similarly, $ K_0^{\Gamma\rtimes \mathbb Z}(\underline{E}(\Gamma\rtimes\mathbb Z)) \cong  K_0^{\Gamma}(\underline{E}\Gamma)  \oplus  K_1^{\Gamma}(\underline{E}\Gamma).$
Moreover, the Baum-Connes assembly map respects this direct sum decomposition. 
Now the map  
\[ \iota_\ast : K_0(C_r^\ast(\Gamma)) \to K_0(C_r^\ast(\Gamma\rtimes \mathbb Z)) \cong K_0(C_r^\ast(\Gamma)) \oplus K_1(C_r^\ast(\Gamma))\]
is simply $ [p] \mapsto ([p] ,  [0]). $
Since $\iota_\ast(\ind_\Gamma (\diffeo))$ lies in the image of the Baum-Connes assembly map 
$\mu: K_0^{\Gamma\rtimes \mathbb Z}(\underline{E}(\Gamma\rtimes\mathbb Z)) \to K_0(C^\ast_r(\Gamma\rtimes \mathbb Z)),$
it follows that $\ind_\Gamma(\diffeo)$ lies in the image of the Baum-Connes assembly map
\[ \mu: K_0^{\Gamma}(\underline{E}\Gamma) \to K_0(C^\ast_r(\Gamma)). \]
\end{proof}

\begin{corollary}\label{cor:main}
Assume that $\diffeo_\ast = \textup{Id}: K_i(C_r^\ast(\Gamma)) \to K_i(C_r^\ast(\Gamma))$
and in addition that the strong Novikov conjecture holds for $\Gamma$.
Then $ \ind_\Gamma(\diffeo ) $ lies in the image of the Baum-Connes assembly map $ \mu: K_0^{\Gamma}(\underline{E}\Gamma) \to K_0(C^\ast_r(\Gamma)).$
\end{corollary}
\begin{proof}
By Pimsner-Voiculescu exact sequence\footnote{In the real case, the Pimsner-Voiculescu exact sequence has $24$ terms instead.}, we have 
\[  \xymatrix{ K_0(C_r^\ast(\Gamma)) \ar[r]^{1 - \diffeo_\ast } & K_0(C_r^\ast(\Gamma)) \ar[r]^-{\iota_\ast} &  K_0(C_r^\ast(\Gamma\rtimes \mathbb Z)) \ar[d]^{\partial_0} \\
K_1(C_r^\ast(\Gamma\rtimes\mathbb Z))\ar[u]^{\partial_1}  & K_1(C_r^\ast(\Gamma)) \ar[l]_-{\iota_\ast } &  K_1(C_r^\ast(\Gamma)) \ar[l]_{1 - \diffeo_\ast}}\]
 Similarly, we have the six-term exact sequence 
\[  \xymatrix{ K_0^\Gamma(\underline{E}\Gamma ) \ar[r]^{1 - \diffeo_\ast } & K_0^\Gamma(\underline{E}\Gamma) \ar[r]^-{\iota_\ast} &  K_0^{\Gamma\rtimes \mathbb Z}(\underline{E}(\Gamma\rtimes \mathbb Z)) \ar[d]^{\partial_0}\\
K_1^{\Gamma\rtimes \mathbb Z}(\underline{E}(\Gamma\rtimes \mathbb Z))\ar[u]^{\partial_1}  & K_1^\Gamma(\underline{E}\Gamma ) \ar[l]_-{\iota_\ast } &  K_1^\Gamma(\underline{E}\Gamma ) \ar[l]_{1 - \diffeo_\ast}}\]
Moreover, the Baum-Connes assembly map is natural with respect to these exact sequences. So by our assumption that $\diffeo_\ast = \textup{Id}: K_i(C_r^\ast(\Gamma)) \to K_i(C_r^\ast(\Gamma))$,  we have the following commutative diagram of short exact sequences:
\[ \xymatrix{ 
0\ar[r] & K_0^\Gamma(\underline{E}\Gamma) \ar[r]^-{\iota_\ast} \ar[d]^\mu &  K_0^{\Gamma\rtimes \mathbb Z}(\underline{E}(\Gamma\rtimes \mathbb Z)) \ar[d]^\mu \ar[r]^-{\partial_0} & K_1^\Gamma(\underline{E}\Gamma )  \ar[d]^\mu \ar[r] & 0\\ 
0 \ar[r] & K_0(C_r^\ast(\Gamma)) \ar[r]^-{\iota_\ast} &  K_0(C_r^\ast(\Gamma\rtimes \mathbb Z)) \ar[r]^-{\partial_0}  & K_1(C_r^\ast(\Gamma))  \ar[r] & 0 \\
} 
\]
By Theorem $\ref{thm:hc}$ above, we have $ \iota_\ast(\ind_\Gamma(\diffeo)) = \ind (D_{\Gamma\rtimes \mathbb Z}).$ Since $\ind(D_{\Gamma\rtimes \mathbb Z})$ lies in the image of the Baum-Connes assembly map $ \mu: K_0^{\Gamma\rtimes \mathbb Z}(\underline{E}(\Gamma\rtimes \mathbb Z)) \to K_0(C^\ast_r(\Gamma\rtimes \mathbb Z)),$ there is  an element $P\in K_0^{\Gamma\rtimes \mathbb Z}(\underline{E}(\Gamma\rtimes \mathbb Z))$ such that $\mu(P) = \ind(D_{\Gamma\rtimes \mathbb Z})$. Notice that 
$ \mu \circ \partial_0 (P)  =  \partial_0\circ \mu (P) = 0.$
Since the strong Novikov conjecture holds for $\Gamma$, that is, the map  $\mu : K_i^\Gamma(\underline{E}\Gamma )  \to  K_i(C_r^\ast(\Gamma))$
is injective, it follows that $ \partial_0 (P) = 0$. Therefore, there exist an element $Q \in K_0^\Gamma(\underline{E}\Gamma)$ such that $ \iota_\ast( Q)  = P.$  The following diagram shows how all these  elements are related under various maps: 

\[ \xymatrix{ 
   Q   \ar@{|->}[r]^-{\iota_\ast}  \ar@{|->}[d]^\mu & P   \ar@{|->}[r]^-{\partial_0 } \ar@{|->}[d]^\mu  & \partial_0 (P)  \ar@{|->}[d]^\mu \\
  \ind_\Gamma(\diffeo)   \ar@{|->}[r]^-{\iota_\ast}  & \iota_\ast(\ind_\Gamma(\diffeo)) = \ind(D_{\Gamma\rtimes \mathbb Z}) \ar@{|->}[r]^-{\partial_0} & 0   
} 
\]
In particular, we see that $\iota_\ast \circ \mu (Q) = \mu \circ \iota_\ast (P)  = \ind (D_{\Gamma\rtimes \mathbb Z}) = \iota_\ast(\ind_\Gamma(\diffeo)).$
Now since $\diffeo_\ast = \textup{Id}$, it follows immediately that $\mu(Q) = \ind_\Gamma (\diffeo) $. This finishes the proof.
\end{proof}

\begin{remark}
The above corollaries have their counterparts for all other dimensions and for both the real and the complex cases, which are essentially proved by the same arguments as above. 
\end{remark}

\appendix

\section{Technical lemmas}\label{appendix}
In this appendix, we prove some standard estimates for pseudodifferential operators with coefficients in $\mathcal A$-bundles, where $\mathcal A$ is an arbitrary real or complex $C^\ast$-algebra. In particular, we prove G\aa rding's inequality in this setting. We would like to point out that all the estimates take values in $\mathcal A$ rather than $\mathbb R $ (or $\mathbb C$). 

Let $X$ be a compact Riemannian manifold. Let $\mathcal A$ be a $C^\ast$-algebra and $\mathcal V$ an $\mathcal A$-bundle over $X$. We denote the $\mathcal A$-valued inner product on $\mathcal V$ by $\langle \, ,\,\rangle $ and define 
\[  \lcn \sigma, \eta\rcn_{s} = \int_X \langle (1+ \Delta)^s\sigma(x), \eta(x)\rangle dx, \]
where $\Delta$ is the Laplacian operator. Then the Soblev space $\vh^s(X, \mathcal V)$ is the completion of $\Gamma^\infty(X, \mathcal V)$ under the norm $\lcn \, ,  \rcn_s$. Notice that $\vh^s(X, \mathcal V)$ is a Hilbert module over $\mathcal A$.  Equivalently, $\lcn \, , \,\rcn_s$ can also be defined through Fourier transform  as in the classical case. In the following, we adopt the convention that the measure on $\mathbb R^n$ is the Lebesgue measure with an additional normalizing factor $(2\pi)^{-n/2}$. 

\begin{lemma}\label{lemma:bbd}
Let $T$ be a pseudodifferential operator of order $n$ 
\[ T : \Gamma^\infty(X, \mathcal V ) \to \Gamma^\infty(X, \mathcal V). \] 
Then 
\[ \lcn T\sigma , T\sigma\rcn_{s-n} \leq C \lcn \sigma, \sigma\rcn_{s}  \]
for all $\sigma\in \vh^s(X, \mathcal V)$.
\end{lemma}
\begin{proof}
The Fourier transform of $T\sigma$ is given by 
\[ \widehat{T\sigma}(\zeta) = \int e^{i\langle x,\, \xi-\zeta\rangle } t(x, \xi) \hat\sigma(\xi) d\xi dx  \]
where $t(x, \xi)$ is the symbol of $T$.
Define 
\[ q(\zeta, \xi) = \int e^{-i\langle x, \, \zeta\rangle } t(x, \xi) dx. \]
Then 
\[  \widehat{T\sigma}(\zeta)  = \int q(\zeta - \xi, \xi) \hat \sigma(\xi) d\xi. \] 
Define 
\[ K(\zeta, \xi) = q(\zeta -\xi, \xi) (1 + |\xi|)^{-s}(1 + |\zeta|)^{s-n}. \]
We now prove that 
\[  \lcn T\sigma, \eta \rcn_{s-n} + \lcn \eta,  T\sigma \rcn_{s-n} \leq C_0\left(\lambda  \lcn \sigma, \sigma \rcn_s + \frac{1}{\lambda}\lcn \eta, \eta\rcn_{s-n}\right)\] 
for all $\eta\in \vh^{s-n}(X, \mathcal V)$ and for all $\lambda > 0$. Here $C_0$ is some fixed constant independent of $\sigma$ and $\eta$.  

Without loss of generality, we reduce the proof to the case where $\mathcal V$ is the trivial $\mathcal A$-bundle $X\times \mathcal A$ and $K(\zeta, \xi)$ is  positive self-adjoint for all $\xi$ and $\zeta$. We have $ K(\xi, \zeta) = (K^{1/2} (\xi, \zeta) )^2 $. 

Notice that for $a, b\in \mathcal A$, we have 
\[  a^\ast b + a b^\ast \leq \lambda a^\ast a + \frac{1}{\lambda} b^\ast b \]
for all $\lambda >0$. It follows that 
\begin{align*}
& \lcn T\sigma, \eta \rcn_{s-n} + \lcn \eta,  T\sigma \rcn_{s-n}\\
& = \int \hat \sigma^\ast (\xi) q^\ast(\zeta - \xi, \xi)  \hat\eta (\zeta) (1+|\zeta|)^{2(s-n)} d\xi d\zeta \\
& \quad + \int  \hat\eta^\ast (\zeta)  q(\zeta - \xi, \xi) \hat \sigma(\xi) (1+|\zeta|)^{2(s-n)} d\xi d\zeta \\
&  = \int \hat \sigma^\ast (\xi)  K^{1/2} (\zeta, \xi)  (1+|\xi|)^{s}  K^{1/2} (\zeta, \xi) \hat\eta (\zeta) (1+|\zeta|)^{s-n} d\xi d\zeta \\ 
& \quad + \int  \hat\eta^\ast (\zeta)  K^{1/2} (\zeta, \xi) (1+|\zeta|)^{s-n}  K^{1/2} (\zeta, \xi)  \hat \sigma(\xi) (1+|\xi|)^{s}  d\xi d\zeta \\
& \leq  \lambda \int \hat \sigma^\ast (\xi)   K(\zeta, \xi) \hat \sigma(\xi)  (1+|\xi|)^{2s}d\xi d\zeta \\
& \quad + \frac{1}{\lambda} \int  \hat\eta^\ast (\zeta)   K(\zeta, \xi) \hat \eta(\zeta)  (1+|\zeta|)^{2(s-n)} d\xi d\zeta \\
& \leq \lambda  \int  \| K(\zeta, \xi) \| \hat \sigma^\ast (\xi) \hat \sigma(\xi)  (1+|\xi|)^{2s}d\xi d\zeta \\
& \quad +  \frac{1}{\lambda} \int  \| K(\zeta, \xi) \| \hat\eta^\ast (\zeta)  \hat \eta(\zeta)  (1+|\zeta|)^{2(s-n)} d\xi d\zeta. 
\end{align*}

By standard estimates from (classical) pseudodifferential calculus, we have
\[ \int \| K(\zeta, \xi) \| d\xi \leq C_0 \quad \textup{and}  \int \| K(\zeta, \xi) \| d\zeta \leq C_0 \]
for some constant $C_0$. It follows that 
\[  \lcn T\sigma, \eta \rcn_{s-n} + \lcn \eta,  T\sigma \rcn_{s-n} \leq C_0\left(\lambda  \lcn \sigma, \sigma \rcn_s + \frac{1}{\lambda}\lcn \eta, \eta\rcn_{s-n}\right)\] 
for all $\eta\in \vh^{s-n}(X, \mathcal V)$ and for all $\lambda > 0$. The proof is finished by choosing $\lambda = C_0$ and $\eta = T\sigma$.
\end{proof}

\begin{lemma}\label{lemma:garding}
Let $P$ be an elliptic  pseudodifferential operator of order $d$ 
\[ P : \Gamma^\infty(X, \mathcal V ) \to \Gamma^\infty(X, \mathcal V). \] 
Then 
\[ \lcn \sigma , \sigma\rcn_{d} \leq C  \lcn \sigma, \sigma \rcn_0 + C \lcn P\sigma, P\sigma\rcn_{0}\]
for all $\sigma\in \vh^d(X, \mathcal V)$.
\end{lemma}
\begin{proof}
Let $Q$ be a parametrix of $P$, that is, $1 - QP$ and $1- PQ$ are smoothing operators. It follows from the previous lemma that 
\begin{align*}
\lcn \sigma, \sigma \rcn_d  &= \lcn (1-QP)\sigma + QP \sigma \, , \, (1- QP)\sigma + QP\sigma\rcn_d \\
& \leq 2 \lcn (1 - QP) \sigma, (1-QP)\sigma \rcn_d + 2 \lcn QP\sigma, QP \sigma\rcn_d\\
& \leq C_1 \lcn \sigma, \sigma \rcn_0 + C_2 \lcn P\sigma, P \sigma\rcn_0
\end{align*}
since $1 - QP$ is a smoothing operator and $Q$ has order $-d$. This finishes the proof. 
\end{proof}

Let $\mathcal V$ be an $\mathcal A$-bundle over $\mathbb R^n$. Consider $\mathbb R^n  = \mathbb R^{n-k}\times \mathbb R^k$ with coordinates $y\in \mathbb R^{n-k}$, $z\in \mathbb R^k$ and dual coordinates $\xi, \zeta$.  We define the restriction map $R: \Gamma^\infty(\mathbb R^n, \mathcal V) \to \Gamma^\infty(\mathbb R^{n-k}, \mathcal V) $ by 
\[ R\sigma (y)  = \sigma(y, 0). \]
We have the following lemma which generalizes the corresponding classical result (cf. \cite[Theorem~6.9]{GF95}) to the case of $\mathcal A$-bundles. 
\begin{lemma}
If $s> k/2$, then we have 
\[ \lcn R\sigma, R\sigma\rcn_{s-k/2} \leq C \lcn \sigma, \sigma \rcn_s\]
for all $\sigma\in \vh^s(\mathbb R^n, \mathcal V)$. 
\end{lemma}
\begin{proof}
Without loss of generality, we assume $\mathcal V = \mathbb R^n \times \mathcal A$. It suffices to show that 
\[ \lcn R\sigma, R\sigma\rcn_{s-k/2} \leq C \lcn \sigma, \sigma \rcn_s\]
for all Schwartz sections $\sigma$. Notice that
\[ \int e^{i\langle \xi, y\rangle } \widehat{R\sigma}(\xi) d\xi = R\sigma(y) = \sigma(y, 0) = \iint e^{i\langle \xi, y\rangle} \hat \sigma(\xi, \zeta) d\xi d\zeta \]
for all $y\in \mathbb R^{n-k}$. So $\widehat{R\sigma}(\xi) = \int \hat\sigma(\xi, \zeta)d\zeta$.
 
It follows that 
\begin{align*}
& 2 \widehat{R\sigma}(\xi) \widehat{R\sigma}^\ast(\xi) \\
& = \iint \hat\sigma(\xi, \zeta_1) \hat\sigma^\ast(\xi, \zeta_2)d\zeta_1 d\zeta_2  +  \iint \hat\sigma(\xi, \zeta_2) \hat\sigma^\ast(\xi, \zeta_1)d\zeta_1 d\zeta_2 \\
& = \iint a(\xi, \zeta_1, \zeta_2) b^\ast(\xi, \zeta_1, \zeta_2) + b(\xi, \zeta_1, \zeta_2) a^\ast(\xi, \zeta_1, \zeta_2) d\zeta_1 d\zeta_2 \displaybreak[1]\\
& \leq \iint a(\xi, \zeta_1, \zeta_2) a^\ast(\xi, \zeta_1, \zeta_2)  + b(\xi, \zeta_1, \zeta_2) b^\ast(\xi, \zeta_1, \zeta_2) d\zeta_1 d\zeta_2\\
& = \iint \sigma(\xi, \zeta_1) \sigma^\ast(\xi, \zeta_1) (1 + |\xi|+ |\zeta_1|)^{2s} (1 + |\xi| + |\zeta_2|)^{-2s} d\zeta_1 d\zeta_2 \\
& \quad +\iint \sigma(\xi, \zeta_2) \sigma^\ast(\xi, \zeta_2) (1 + |\xi|+ |\zeta_2|)^{2s} (1 + |\xi| + |\zeta_1|)^{-2s} d\zeta_1 d\zeta_2, 
\end{align*}
where we denote by
\[a(\xi, \zeta_1, \zeta_2) = \hat \sigma(\xi, \zeta_1) (1 + |\xi|+ |\zeta_1|)^{s} (1 + |\xi| + |\zeta_2|)^{-s},\]
\[b(\xi, \zeta_1, \zeta_2) = \hat \sigma(\xi, \zeta_2) (1 + |\xi|+ |\zeta_2|)^{s} (1 + |\xi| + |\zeta_1|)^{-s}.\]
Notice that 
\[ \int (1 + |\xi| + |\zeta|)^{-2s} d\zeta = C ( 1 + |\xi|)^{k-2s} \]
for some constant $C$. Therefore, we have 
\begin{align*}
& 2 \widehat{R\sigma}(\xi) \widehat{R\sigma}^\ast(\xi) \leq 2 C ( 1 + |\xi|)^{k-2s} \iint \sigma(\xi, \zeta) \sigma^\ast(\xi, \zeta) (1 + |\xi|+ |\zeta|)^{2s} d\zeta.
\end{align*}
Equivalently, 
\[  \widehat{R\sigma}(\xi) \widehat{R\sigma}^\ast(\xi) ( 1 + |\xi|)^{2s-k} \leq  C \iint \sigma(\xi, \zeta) \sigma^\ast(\xi, \zeta) (1 + |\xi|+ |\zeta|)^{2s} d\zeta. \]
Now the lemma follows by integrating both sides with respect to $\xi$.
 \end{proof}

\begin{corollary}
Let $\Omega$ be a bounded domain in $\mathbb R^n$ with $C^\infty$ boundary $\partial \Omega$ and $\ell \geq 1$. Define the restriction map 
\[ R(\sigma) = \sigma|_{\partial \Omega} : C^\ell(\overline\Omega, \mathcal V) \to C^\ell(\partial\Omega, \mathcal V).\]
Then 
\[ \lcn R\sigma, R\sigma\rcn_{\ell-1/2} \leq C \lcn \sigma, \sigma \rcn_\ell\]
for all $\sigma\in \vh^\ell(\Omega, \mathcal V)$.

\end{corollary}


\end{document}